\documentclass[]{amsart}

\usepackage[Symbol]{upgreek}

\usepackage{mathdots}

\usepackage{amssymb}
\usepackage{amscd}
\usepackage{epic}
\usepackage{mathrsfs}
\usepackage{accents}
\usepackage{stmaryrd}
\usepackage{hyperref} 

\usepackage{pgfplots, pgfplotstable}
\pgfplotsset{compat=1.18}
\usepgfplotslibrary{fillbetween}
\usetikzlibrary{patterns}

\usepackage{xr}
\usepackage{etoolbox} 

\numberwithin{equation}{section}

\newcommand{\bbold}{\mathbb}

\renewcommand\Re{\operatorname{Re}}
\renewcommand\Im{\operatorname{Im}}

\def\R { {\bbold R} }
\def\Q { {\bbold Q} }

\def\Co { {\bbold C} }
\def\N { {\bbold N} }

\def \ex{\operatorname{e}}

\renewcommand\epsilon{\varepsilon}

\def \<{\langle}
\def \>{\rangle}

\def \hat {\widehat}

\def \supp {\operatorname{supp}}

\def \((  {(\!(}
\def \)) {)\!)}

\DeclareMathSymbol{\precequ}{\mathrel}{symbols}{"16}
\DeclareMathSymbol{\succequ}{\mathrel}{symbols}{"17}

\newcommand{\claim}[2][\!\!]{\medskip\noindent {\it Claim #1}\/: {\it #2}\medskip}

\newtheorem{theorem}{Theorem}[section]

\newtheorem{lemma}[theorem]{Lemma}
\newtheorem{prop}[theorem]{Proposition}
\newtheorem{cor}[theorem]{Corollary}

\newtheorem*{theoremUnnumbered}{Theorem}

\newtheorem{manualtheoreminner}{Theorem}
\newenvironment{manualtheorem}[1]{%
  \IfBlankTF{#1}
    {\renewcommand{\themanualtheoreminner}{\unskip}}
    {\renewcommand\themanualtheoreminner{#1}}%
  \manualtheoreminner
}{\endmanualtheoreminner}

\newtheorem{manualcorinner}{Corollary}
\newenvironment{manualcor}[1]{%
  \IfBlankTF{#1}
    {\renewcommand{\themanualcorinner}{\unskip}}
    {\renewcommand\themanualcorinner{#1}}%
  \manualcorinner
}{\endmanualcorinner}

\theoremstyle{definition}

\theoremstyle{remark}
\newtheorem*{example}{Example}

\newtheorem*{notation}{Notation}

\newtheorem*{remarks}{Remarks}
\newtheorem*{remark}{Remark}

\newcommand{\abs}[1]{\lvert#1\rvert}
\newcommand{\dabs}[1]{\lVert#1\rVert}

\let\oldi\i
\let\oldj\j
\renewcommand\i{\relax\ifmmode{\boldsymbol{i}}\else\oldi\fi}
\renewcommand\j{\relax\ifmmode{\boldsymbol{j}}\else\oldj\fi}

\renewcommand\leq{\leqslant}
\renewcommand\geq{\geqslant}

\renewcommand\le{\leq}
\renewcommand\ge{\geq}

\DeclareMathAlphabet{\mathbf}{OML}{cmm}{b}{it}

\DeclareFontFamily{U}{fsy}{}
\DeclareFontShape{U}{fsy}{m}{n}{<->s*[.9]psyr}{}
\DeclareSymbolFont{der@m}{U}{fsy}{m}{n}
\DeclareMathSymbol{\der}{\mathord}{der@m}{182}

\DeclareSymbolFont{der@m}{U}{fsy}{m}{n}
\DeclareMathSymbol{\derdelta}{\mathord}{der@m}{100}

\DeclareSymbolFont{imag@m}{OT1}{cmr}{m}{ui}
\DeclareMathSymbol{\imag}{\mathord}{imag@m}{105}

\DeclareFontFamily{OMS}{smallo}{}
\DeclareFontShape{OMS}{smallo}{m}{n}{<->s*[.65]cmsy10}{}
\DeclareSymbolFont{smallo@m}{OMS}{smallo}{m}{n}
\DeclareMathSymbol{\smallo}{\mathord}{smallo@m}{79}

\DeclareFontFamily{OMS}{largerdot}{}
\DeclareFontShape{OMS}{largerdot}{m}{n}{<->s*[.8]cmsy10}{}
\DeclareSymbolFont{largerdot@m}{OMS}{largerdot}{m}{n}
\DeclareMathSymbol{\largerdot}{\mathord}{largerdot@m}{15}

\DeclareMathSymbol{\llambda}{\mathord}{der@m}{108}
\DeclareMathSymbol{\rrho}{\mathord}{der@m}{114}

\newcommand{\equationqed}[1]{\[\pushQED{\qed}#1 \qedhere\popQED\]\let\qed\relax}
\newcommand{\alignqed}[1]{\begin{align*}\pushQED{\qed} #1 \qedhere\popQED\end{align*}\let\qed\relax}

\makeatletter
\newcommand{\dminus}{\mathbin{\text{\@dminus}}}

\newcommand{\@dminus}{%
  \ooalign{\hidewidth\raise1ex\hbox{\bf.}\hidewidth\cr$\m@th-$\cr}%
}
\makeatother

\def \Cc{\mathcal{C}}
\def \C{\mathcal{C}}

\DeclareRobustCommand{\stirlingsecond}{\genfrac\{\}{0pt}{}}

\begin{document}

\title{Whitney Approximation: domains and bounds}
 
\author[Aschenbrenner]{Matthias Aschenbrenner}
\address{Kurt G\"odel Research Center for Mathematical Logic\\
Universit\"at Wien\\
1090 Wien\\ Austria}
\email{matthias.aschenbrenner@univie.ac.at}

\date{August 2025}

\begin{abstract} We investigate properties of   holomorphic extensions in the one-vari\-able case of Whitney's Approximation Theorem on intervals. 
Improving a result of Gauthier-Kienzle,
we construct tangentially approximating functions which extend holomorphically to  domains of optimal size. For approximands on unbounded closed intervals, we also
bound the growth of   holomorphic extensions, in the spirit of Arakelyan, Bernstein,
Keldych, and Kober.
\end{abstract}


\maketitle

\section*{Introduction}
 
\noindent
The Weierstrass Approximation Theorem regarding the approximation of a continuous 
function $f$ on a compact interval by polynomials was generalized to the complex domain
by Mergelyan~\cite{Mergelyan} in 1951, who replaced the compact interval by a compact subset $K$ of the 
complex plane having connected complement with the additional (necessary) condition
on $f$ that it be not only continuous on $K$ but also holomorphic on the interior of $K$.
(See \cite[Chapter~III, \S{}2]{Gaier} or \cite[Chapter~20]{Rudin} for a proof.)
This powerful theorem  sparked
the development of the subject of complex approximation theory in the second half of the 20th century,
and has had numerous applications in many branches of complex analysis.
Carleman \cite{Carleman}   had previously (1927) found another analogue of the Weierstrass theorem, in which he 
replaced the compact interval by the real line and polynomials by entire functions.
(Actually,  Carleman's theorem can be deduced without much effort
from Mergelyan's, cf.~\cite[proof of Theorem~8]{Pinkus}.)
Staying with the approximation of functions on the real line, 
the present paper is devoted to a generalization of Carleman's theorem due to Whitney (1934).

To explain Whitney's result, we fix some conventions which will remain in force throughout
the rest of this paper:  we let $a$, $b$, $s$, $t$ range over $\R$, $m$, $n$ over $\N=\{0,1,2,\dots\}$, $r$ over~${\N\cup\{\infty\}}$, and~$z$ over~$\Co$.
Let   $I=(\alpha,\beta)$, where $-\infty\le\alpha<\beta\le\infty$, be an open interval.  
Each $\C^r$-function~${I\to\R}$ and its derivatives can be arbitrarily well approximated by an analytic function and its derivatives; more precisely:

\begin{theoremUnnumbered}[Whitney]\label{thm:WAP}
Let $f\in\C^r(I)$  and $\varepsilon,\rho\in\C(I)$ be such that $\varepsilon>0$ and~$r\ge\rho\ge 0$ on~$I$.  Then there is
a $g\in\C^\omega(I)$ such that  
\begin{equation}\label{eq:WAP}\tag{1}
\abs{(f-g)^{(n)}(t)}<\varepsilon(t)\qquad\text{ for all~$t\in I$ and~$n\le \rho(t)$.}
\end{equation}
\end{theoremUnnumbered}

\noindent
This theorem is implicit in the proof of a (more general) lemma in Whitney's  seminal~1934 paper on extending multivariate differentiable functions~\cite{Whitney}.
In \cite[Appendix~A]{fgh-an} we included a self-contained proof of this fact, which 
makes it apparent that there is an open subset $U=U_I$ of $\Co$ which contains~$I$, independent of~$f$,
such that
the   function~$g$ in the conclusion
of the theorem can be taken as the restriction of a holomorphic function~${U\to\Co}$.
 If~$I$ is bounded, then the open set~$U$
  obtained through the    proof of Whitney's theorem as in loc.~cit.~is also bounded. (See \cite[Corollary~A.4]{fgh-an}.)
However, if~$I=\R$, then this procedure yields~$U=\Co$,   so~$g$ can be taken as the restriction of an entire function. Independently of Whitney,
this fact  had also been discovered by Hoischen~\cite{Hoischen}; for~$r=0$, this is Carleman's 1927 theorem~\cite{Carleman} mentioned above,   and for~$r=1$ this was shown by Kaplan~\cite{Kaplan} in 1955. (See~\cite{Pinkus}
for the history of such approximation theorems, going back to Weierstrass.)
By a   theorem of Gauthier and Kienzle~\cite[Theorem~1.3]{GK}, in the case~${r<\infty}$   one can always take~$U=I\cup(\Co\setminus\R)$.
Our first main result strengthens this fact by showing that it is possible to   choose~$U=\Co\setminus\{\alpha,\beta\}$ (even when $r=\infty$):

\begin{manualtheorem}{1}\label{thm:WAP GK}
For all $f$, $\varepsilon$, $\rho$ as in Whitney's Theorem,  there is
a $g\in\C^\omega(I)$ satisfying~\eqref{eq:WAP} which
extends to a holomorphic function $\Co\setminus\{\alpha,\beta\}\to\Co$. 
\end{manualtheorem}

\noindent
The proof of this theorem, given in Section~\ref{sec:WAP GK} below, involves a simple reduction to Whitney's theorem in the case $I=\R$, by employing   compositions with suitably chosen rational functions and some estimates involving Fa\`a di Bruno's Formula. In this section we also generalize
Theorem~\ref{thm:WAP GK}     to other kinds of intervals: 

\begin{manualcor}{1}\label{cor:WAP GK}
Let $I\subseteq\R$ be connected and $f\in\C^r(I)$ \textup{(}i.e., $f$ extends to a function in $\C^r(J)$ for some open interval $J\supseteq I$\textup{)},
and let $\varepsilon,\rho\in\C(I)$ be such that~$\varepsilon>0$ and~$r\ge\rho\ge 0$ on~$I$.  Then there is
a $g\in\C^\omega(I)$ satisfying~\eqref{eq:WAP} which extends to a holomorphic function $\Co\setminus \operatorname{fr}(I)\to\Co$,
where $\operatorname{fr}(I)=\operatorname{cl}(I)\setminus I$ is the frontier of $I$. 
\end{manualcor}

\noindent
The second topic of this note are versions of Whitney's theorem which include explicit bounds on the growth of the holomorphic extension of $g$ to $U$.
The following example, due to Arakelyan~\cite[pp.~14--15]{Arakelyan}, indicates why   in the case 
of Carleman's Theorem~(${I=\R}$ and~$r=0$) one cannot expect to bound the growth of the entire extension of~$g$ solely from information about that  of $f$:

\begin{example}
Let $\beta\in\C(\R)$ be even and strictly increasing on $\R^\ge$ such that $\beta(0)=0$ and~${\beta(t)\to\infty}$ as $t\to\infty$,    and let
$\varepsilon\in\C(\R)$ be bounded with
$\varepsilon>0$.
Then the continuous function $f\colon\R\to\R$, $f(t):=\varepsilon(t)\cos\beta(t)$ is also bounded. 
Call an entire function~$g$   {\it real}\/ if $g(\R)\subseteq\R$.
Suppose now that~$g$ is a real entire function such that~$\abs{f(t)-g(t)}<\varepsilon(t)$ for each $t$. Then for each $t$ we have
$$-1+\cos\beta(t) < \frac{g(t)}{\varepsilon(t)} < 1+\cos\beta(t),$$
so between any two reals  $a$, $b$ with $\cos\beta(a)=1$, $\cos\beta(b)=-1$ there is a zero of~$g$. This implies that each interval $[-t,t]$ ($t\ge 0$)
contains~$\ge (2/\pi)\beta(t)-2$ zeros of $g$. Using Jensen's Formula~\cite[1.2.1]{Boas} we obtain a $C\in\R$ such that
$$\log\, \dabs{g}_t \ge \frac{2}{\pi} \int_1^t \frac{\beta(s)}{s}\,ds - 2\log t+C\qquad\text{for each $t\ge 1$.}$$
Here and below $\dabs{g}_t:=\sup\big\{\abs{g(z)}:\abs{z}\le t\big\}$ for $t\ge 0$.
\end{example}

\noindent
However, Bernstein~\cite{Bernstein} and Kober~\cite{Kober}
discovered that  in this situation, for constant error function~$\varepsilon$, the growth of
the entire function~$g$ can be limited, {\it provided}\/ that~$f$ is not only assumed to be bounded, but also uniformly continuous. (See 
\cite[Theorem~2.1]{DQR}.)
 Keldych~\cite{Keldych} proved a similar result, assuming that $f$ is $\C^1$ and  the asymptotics of both~$f$ and~$f'$ are known
(see~\cite[p.~163]{Gaier}), and
this was refined and extended to more general error functions   by Arakelyan~\cite{Arakelyan63, Arakelyan}.
Our second theorem provides a partial generalization of these facts to the case~$r>0$, under some natural restrictions on~$\varepsilon$,~$\rho$. 
To formulate it, we introduce some notation.
Let~$f\in\C^m(\R)$. For $t\ge 0$ and
  $\rho\in\R$ with~$0\le\rho\le m$    put
$$\dabs{f}_{t;\,\rho}  :=  \sup\big\{ \abs{f^{(n)}(s)}:\abs{s}\le t,\ n\le \rho\big\}\in\R^{\ge}, \qquad \dabs{f}_t:=\dabs{f}_{t;\,0}.$$
Note that~$\dabs{f}_{t_1;\,\rho_1} \le \dabs{f}_{t_2;\,\rho_2}$ if $0\le t_1\le t_2$ and $0\le\rho_1\le\rho_2\le m$.
By convention  we set $\infty+1:=\infty$, and for $\phi\colon \R^{\ge a}\to\R$ we let $\Delta\phi\colon \R^{\ge a}\to\R$ denote the difference function~$t\mapsto \phi(t)-\phi(t+1)$. (If $\phi$ is strictly decreasing, then $\Delta\phi(t)>0$ for~$t\ge a$, and if $\phi$ is convex, then $\Delta\phi$ is decreasing: see Section~\ref{sec:prelims}.)
 As usual  $\log^+t:=\max\{0,\log t\}$, where~$\log t:=-\infty$ for $t\le 0$.

\begin{manualtheorem}{2}\label{thm:WAP bd, 1}
For each~${f\in \C^{r+1}(\R)}$ and~$\varepsilon,\rho\in\C(\R^{\ge})$   where       $\varepsilon>0$ is strictly decreasing and convex with~${\varepsilon(t)\to 0}$ as~$t\to \infty$, and $\rho$ with  $r\ge \rho\ge 0$ is increasing,
there are some~$C,D\in\R^>$ and a real entire function $g$   such that
$$\abs{(f-g)^{(n)}(t)} < \varepsilon(\abs{t})\quad\text{ for   $n\le\rho(\abs{t})$,}$$ 
and  
$$\dabs{g}_t \le \exp\big( C\cdot s^2\cdot\big(1+\log^+ \dabs{f}_{s+3\sqrt{2}}+\lambda(s+3\sqrt{2}) \big)\big)\quad\text{for $t\ge 0$, $s=\sqrt{2}t+1$,}$$
where 
$$\lambda(s):= \big(D(\rho(s)+1)\big)^{D(\rho(s)+1)s} \cdot \left(\frac{\dabs{f}_{s;\,\rho(s)+1}}{\Delta\varepsilon(s)} \right)^3.$$
Here $C$ can be chosen independently of  $\rho$, and $D$   independently of $f$, $\varepsilon$, $\rho$, $r$.
\end{manualtheorem}

\noindent
The proof of this theorem is given in Section~\ref{sec:WAP with bds}, and is obtained by making the already quite constructive argument
of Whitney completely explicit. This requires   some bounds for derivatives of bump functions and for 
the approximation of functions of bounded support by Weierstrass transforms,   computed in Sections~\ref{sec:bump fns} and~\ref{sec:Weierstrass transforms} after some preliminaries in Section~\ref{sec:prelims}.
In the rest of this introduction we collect some applications of the previous theorem, and also state a variant  for functions~$f\in\C^{r+1}(\R^{\ge})$.

First, for $r<\infty$, by  Theorem~\ref{thm:WAP bd, 1} applied to the constant function~${\rho(t):=r}$:

\begin{manualcor}{2}\label{cor:WAP bd, 1}
Suppose $r<\infty$. Then for each $f\in\C^{r+1}(\R)$  and
$\varepsilon\in\C(\R^{\ge})$   such that~$\varepsilon>0$ is strictly decreasing and convex with 
$\varepsilon(t)\to 0$ as $t\to \infty$, there are~$C,D\ge 1$ and a real entire function $g$   such that
$\abs{(f-g)^{(n)}(t)} < \varepsilon(\abs{t})$ for~$n\le r$,
and  
$$\dabs{g}_t \le 
\exp\left( C\cdot s^2\cdot \left(1+\log^+ \dabs{f}_{s} + D^{s} \left(\frac{\dabs{f}_{s;\,r+1}}{\Delta\varepsilon(s)}\right)^3 \right)\right)$$
for  $t\ge 0$ and $s=\sqrt{2}(t+4)$. Here $D$ can be chosen to only depend on $r$.
\end{manualcor}

\noindent
For $r=0$ and certain   $\varepsilon$, Arakelyan~\cite{Arakelyan} has a qualitatively better bound:

\begin{remark}
Let 
$\varepsilon\in\C^1(\R^{\ge})$ be   decreasing  with $\varepsilon>0$, such that $\lim\limits_{t\to\infty} t\varepsilon'(t)/\varepsilon(t)\in\R^{\le}$
exists, and let $f\in\C^{1}(\R)$. Then by \cite[Theorem~6]{Arakelyan} there are a $C\in\R^>$ and a real entire function $g$  such that $\abs{f(t)-g(t)}<\varepsilon(\abs{t})$ for each $t$ and
$$\dabs{g}_t \le \exp\left( C\cdot s\cdot\left(1+\log^+\left(\frac{\dabs{f}_{s}}{\varepsilon(t)}\right)+\frac{\dabs{f'}_{s}}{\varepsilon(t)}\right)\right)\quad\text{if $t\ge 0$, $s=\sqrt{2}(t+4)$.}$$
To facilitate the comparison with the bound in Corollary~\ref{cor:WAP bd, 1}, note that
if in addition~$\varepsilon$ is assumed to be $\C^2$, strictly decreasing, and convex,  
then we have $1/\Delta\varepsilon(s) = O\big( {(s+1)/\varepsilon(s+1)}\big)$ as~$s\to\infty$.
(Lemma~\ref{lem:Deltaphi}.)
\end{remark}

\noindent
Before stating the next corollary, we define some quantities measuring the growth of an entire function $g$ introduced in~\cite{Sato}. 
First recall that except in the case where $g$ is a constant, the function $t\mapsto \dabs{g}_t\colon\R^{\ge}\to\R^{\ge}$
is strictly increasing with $\dabs{g}_t\to\infty$ as $t\to\infty$.
If~$g$ is given by a polynomial of degree $n$ then~$\dabs{g}_t=O(t^n)$ as $t\to\infty$, and
if~$\lambda\in\R^{\ge}$ is such that $\dabs{g}_t=O(t^\lambda)$ as $t\to\infty$, then $g$ is given by a polynomial
of degree at most $\lfloor\lambda\rfloor$.
Suppose now that $g$ is non-constant.
With $\log_m$ denoting the $m$-fold iterated logarithm,
define
$$\lambda_m(g) := \limsup_{t\to\infty} \frac{\log_m \dabs{g}_t}{\log t} \in [0,\infty].$$
If $\lambda_m(g)<\infty$ for some $m$, then $g$ is said to have {\it   finite index,}\/ and in this case the
smallest such $m$ is called the {\it index}\/ of $g$.
Note that if $g$ has finite index $m$, then~${m\ge 1}$. 
By convention, constant functions $\Co\to\Co$ have index $0$.
The entire functions of index~$1$ are exactly the non-constant  polynomial functions. The
entire functions of index $\le 2$ are also known as the entire functions of finite order, and~
$\lambda_2(g)$ is called the {\it order}\/ of~$g$. (See~\cite[Chapter~2]{Boas}.)
  By~\cite[Theorem~1]{Sato}, if~$g$ has index $m\ge 2$, then
  $\lambda=\lambda_m(g)$ 
  may be computed from the Taylor coefficients~$g_n:=g^{(n)}(0)/n!$ of $g$ at $0$ as follows:
$$\lambda  = \limsup_{n\to\infty}  \frac{n\log_{m-1} n}{-\log \abs{g_n}}.$$ 
For $f\in\C(\R)$  put~$\dabs{f}:=\sup_{t\ge 0} \dabs{f}_t\in [0,\infty]$.
From Corollary~\ref{cor:WAP bd, 1} we obtain:

\begin{manualcor}{3}\label{cor:WAP bd, 2}
Suppose $r<\infty$, and let $f\in\C^{r+1}(\R)$ be such that~$\dabs{f^{(n)}}<\infty$ for~$n\le r+1$,
and   $\varepsilon\in\R^>$. Then there is a real entire function $g$ of index~$\le 3$ such that~$\abs{({f-g})^{(n)}(t)} < \varepsilon$ for~$n\le r$.
\end{manualcor}
\begin{proof}
Consider
$\varepsilon_0\in\C(\R^{\ge})$,
 $\varepsilon_0(t):=\varepsilon/(t+1)$; then $\Delta\varepsilon_0(t)=\frac{\varepsilon}{(t+1)(t+2)}$ for~${t\ge 0}$. Take $C$, $D$, $g$ as in Corollary~\ref{cor:WAP bd, 1} applied to $\varepsilon_0$ in place of $\varepsilon$. 
 Then with $M:=\max\!\big\{1,\dabs{f}_0,\dots,\dabs{f}_{r+1}\big\}$,
 for $t\ge 0$ and $s=\sqrt{2}(t+4)$ we have
 $$\dabs{g}_t \le \exp\left( C\cdot s^2\cdot \left(1+\log M+D^s\left(\frac{M(s+1)(s+2)}{\varepsilon}\right)^3\right)\right).$$
 For suitable $E\in\R^{\ge 1}$ (only depending on $C$, $D$, $\varepsilon$, $M$) we thus have $\dabs{g}_t \le \exp(E^s)$ for all sufficiently large $t\ge 0$, and this yields $\lambda_3(g)\le 1$.
\end{proof}

\begin{remarks}
Let $f\in\C^1(\R)$ and $M\in\R^{\ge 1}$ be such that $\dabs{f},\dabs{f'}\le M$, and $\varepsilon\in\R^>$. Then by the remark following Corollary~\ref{cor:WAP bd, 1} (applied to the constant function $\varepsilon$), there is
 a $C\in\R^>$ and a real entire function $g$ such that~$\abs{(f-g)(t)} < \varepsilon$ for all $t$ and
$$\dabs{g}_t \le \exp\left( C\cdot s\cdot\left(1+\log^+\left(\frac{M}{\varepsilon}\right)+\frac{M}{\varepsilon}\right)\right)\quad\text{if $t\ge 0$, $s=\sqrt{2}(t+4)$,}$$
thus $\lambda_2(g)\le 1$. Therefore in Corollary~\ref{cor:WAP bd, 2}, in the case $r=0$   one can replace ``in\-dex~$\le 3$'' by ``order~$\le 1$''.
(This is originally due to Bernstein~\cite{Bernstein} and Kober~\cite{Kober}.) It would be interesting to know whether  in the case $r>0$, one 
can improve  ``in\-dex~$\le 3$'' to ``finite order'' in the previous corollary.
\end{remarks}

\noindent
The case $r=\infty$ of Theorem~\ref{thm:WAP bd, 1} leads to an application to rapidly decreasing functions:
Let $f\in\C^\infty(\R)$, and put $\dabs{f}_{m,n}:=\sup\big\{ \abs{t^m f^{(n)}(t)}:t\in\R\big\} \in [0,\infty]$.
 Then $f$ is said to be {\it rapidly decreasing}\/ if~$\dabs{f}_{m,n}<\infty$ for all $m$, $n$;
 equivalently, if for all $m$, $n$ we have $t^m f^{(n)}(t)\to 0$ as~$\abs{t}\to\infty$.
Examples: if $f$ has bounded support, then it is rapidly decreasing;
and for each $\lambda\in\R^>$ and~$m$,     the function~$t\mapsto t^m\ex^{-\lambda t^2}\colon\R\to\R$ is rapidly decreasing.
The  rapidly decreasing functions form a subalgebra $\mathcal S(\R)$ of the $\R$-algebra~$\C^\infty(\R)$.
The collection of sets 
$$\big\{ g\in\mathcal S(\R): \text{$\dabs{f-g}_{m,n}<\varepsilon$ for $m,n\le N$}\big\}\quad\text{where $f\in\mathcal S(\R)$, $\varepsilon\in\R^>$,  and $N\in\N$,}$$
is the basis of a topology on $\mathcal S(\R)$, which makes the $\R$-vector space $\mathcal S(\R)$ into a topological $\R$-vector space  (in fact, a Fr\'echet space), known
as the Schwartz space on~$\R$, which plays an important role in the theory of distributions (cf.~\cite{Treves}).
It is well-known that the set of rapidly decreasing functions which extend to entire functions is dense in $\mathcal S(\R)$; 
see \cite[Theorem~15.5, p.~160]{Treves}.
From Theorem~\ref{thm:WAP bd, 1} we obtain an effective
version of this fact:

\begin{manualcor}{4}\label{cor:rapid dec}
Let $f\in\mathcal S(\R)$, $\varepsilon\in\R^>$, and $N\in\N$. Then there is a real entire function~$g$
such that $g|_{\R}\in\mathcal S(\R)$ and $\dabs{ f - g }_{m,n} \le \varepsilon$ for each $m,n\le N$. Moreover, we can choose $g$ so that  for some $C,D\in\R^{\ge 1}$ we have
$$\dabs{\hat g}_t \le \exp\big( C\cdot s^2\cdot\big(1+\log^+ \dabs{f}_{s+3\sqrt{2}}+\lambda(s+3\sqrt{2}) \big)\big)\quad\text{for $t\ge 0$, $s=\sqrt{2}t+1$,}$$
where 
$$\lambda(s):= \big(D\,(r+1)\big)^{D\,(r+1)\,s} \cdot \left( \frac{ (N/\!\ex)^N \dabs{f}_{0,\lceil r+1\rceil}}{\varepsilon} \right)^3 \qquad\text{where $r=  N+s $.}$$
Here $C$ can be chosen independently of $N$, and $D$   independently of $f$, $\varepsilon$, $N$.  
\end{manualcor}
\begin{proof}
Put $\delta:=\varepsilon/(N/\!\ex)^N\in\R^>$ and consider   $\varepsilon_0,\rho_0\in\C(\R^{\ge})$ given by
$\varepsilon_0(t):=\delta\ex^{-t}$ and $\rho_0(t):=N+t$ for $t\ge 0$. 
Take $g$ as in Theorem~\ref{thm:WAP bd, 1}
applied to $\varepsilon_0$, $\rho_0$ in place of~$\varepsilon$,~$\rho$. Then for each $m$, $n$ we have
$$\abs{t^m g^{(n)}(t)} \le \abs{t^m f^{(n)}}+\delta \abs{t}^m \ex^{-\abs{t}}\quad\text{if  $\abs{t}\ge n-N$,}$$
thus 
$$\dabs{g}_{m,n}\le \dabs{x^mg^{(n)}}_{\max\{0,n-N\}}+\dabs{f}_{m,n}+\delta (m/\!\ex)^m<\infty.$$
Hence $g|_{\R}\in\mathcal S(\R)$; moreover, if $m,n\le N$ then
$$\abs{t^m(f-g)^{(n)}} \le \delta \abs{t}^m \ex^{-\abs{t}} \le \delta \abs{t}^N\ex^{-\abs{t}} \le \varepsilon\quad\text{for each $t$,}$$
hence $\dabs{f-g}_{m,n}\le\varepsilon$. The rest now follows from Theorem~\ref{thm:WAP bd, 1}.
Note that $\Delta \varepsilon_0(t)=\delta\ex^{-t}(1-\ex^{-1})\ge \epsilon_0(t)/2$ and so
$1/\Delta \varepsilon_0(t) \leq (2/\varepsilon)(N/\!\ex)^N\ex^t$, for each $t\ge 0$.
\end{proof}

\noindent
Finally, in Section~\ref{sec:WAP with bds} we also establish a version of Theorem~\ref{thm:WAP bd, 1} for~$f\in\C^{r+1}(\R^{\ge})$. 
For such $f$, given
$\varepsilon,\rho\in\C(\R^{\ge})$ with $\varepsilon >0$ and $r\ge\rho\ge 0$,
 Corollary~\ref{cor:WAP GK} yields a real entire   function $g$ such that
 $\abs{(f-g)^{(n)}(t)} < \varepsilon(t)$ for $n\le\rho(t)$, $t\ge 0$.
Here we will confine ourselves to constructing a holomorphic approximation of $f$ 
with domain~$U:=\big\{z\in\Co:    \abs{\Im z} < \Re z \big\}$
  while at the same time controlling its growth,  as in Theorem~\ref{thm:WAP bd, 1}, on a proper subset of $U$.
Notation: for $f\in\C^m(\R^{\ge})$, $t\ge 0$, and~$\rho\in [0,m]$    put
$$\dabs{f}_{t;\,\rho}  :=  \sup\big\{ \abs{f^{(n)}(s)}:0\le s\le t,\ n\le \rho\big\}\in\R^{\ge}, \qquad \dabs{f}_t:=\dabs{f}_{t;\,0}.$$

\begin{manualtheorem}{3}\label{thm:WAP bd, 2}
Let $\alpha\in\R^>$, $f\in\C^{r+1}(\R^{\ge})$,
and~$\varepsilon,\rho\in\C(\R^{>})$   where       $\varepsilon>0$ is strictly decreasing and convex with~${\varepsilon(t)\to 0}$ as~$t\to \infty$, and $\rho$ with  $r\ge \rho\ge 0$ is increasing. 
Then  
there are some~$C,D\in\R^>$ and a   holomorphic  $g\colon U\to\Co$  such that $g(\R^>)\subseteq\R$,
$$\abs{(f-g)^{(n)}(t)} < \varepsilon(t)\quad\text{ for   $n\le\rho(t)$ and all $t>0$,}$$ 
and for all $z\in\Co$ with $\Re z>0$ and $(\Im z)^2 \le (\Re z)^2-\alpha$ we have
$$\abs{g(z)} \le \exp\big( C\cdot s\cdot\big(1+\log^+ \dabs{f}_{s}+\lambda(s) \big)\big)\quad\text{for $t\ge\abs{z}$, $s=D(t^2+1)$,}$$
where
$$\lambda(s):=\big(D(\rho(s)+1)\big)^{D(\rho(s)+1)s^2} \cdot \left(\frac{\dabs{f}_{s;\,\rho(s)+1}}{\Delta\varepsilon(s)} \right)^3.$$
Here $C$  can be chosen independently of $\rho$, and $D$ independently of $f$, $\varepsilon$, $\rho$, $r$.
\end{manualtheorem}

 \subsection*{Notations and conventions}
Let~${U\subseteq\R}$ be  open and $\emptyset\neq S\subseteq U$.
For~$f\in\Cc(U)$ we set
$$\dabs{f}_S\ :=\ \sup\big\{ \abs{f(s)}:\, s\in S\big\} \in [0,\infty], $$
so for $f,g\in\Cc(U)$ and $\lambda\in \R$ (and the convention $0\cdot\infty=\infty\cdot 0=0$) we have 
$$\dabs{f+g}_S\le \dabs{f}_S+\dabs{g}_S,\quad \dabs{\lambda f}_S=|\lambda|\cdot\dabs{f}_S,\quad\text{and}\quad
\dabs{fg}_S\leq\dabs{f}_S\dabs{g}_S.$$  
If $\emptyset\neq S'\subseteq S$ then 
$\dabs{f}_{S'}\leq\dabs{f}_S$.
Next,
let~$f\in\Cc^m(U)$. 
We then put
$$\dabs{f}_{S;\,m}\ :=\ \max\big\{ \dabs{f}_S,\dots,\dabs{f^{(m)}}_S\big\}\in [0,\infty].$$
Then again for  $f,g\in\Cc^m(U)$ and $\lambda\in \R$  we have 
$$\dabs{f+g}_{S;m}\le \dabs{f}_{S;m}+\dabs{g}_{S;m},\quad \dabs{\lambda f}_{S;m}=|\lambda|\cdot\dabs{f}_{S;m},$$
and 
\begin{equation}\label{eq:prod norm}
\dabs{fg}_{S;\,m}\  \leq\ 2^m \dabs{f}_{S;\,m}\dabs{g}_{S;\,m}.
\end{equation}
Also note that for $f\in\Cc^m(\R)$, $t\ge 0$, and $\rho\in [0,m]$ we have 
$\dabs{f}_{t;\,\rho}=\dabs{f}_{[-t,t];\,n}$ where $n=\lfloor r\rfloor$.

Let $f\in \Cc^m(U)$. For $U=\R$ we set 
$\dabs{f}_{m}:=\dabs{f}_{\R;\,m}$ and $\dabs{f}:=\dabs{f}_0$.
For~$k\leq m$ and $\emptyset\neq S'\subseteq S\subseteq U$ we have
$\dabs{f}_{S';\,k} \leq \dabs{f}_{S;\,m}$. Moreover, $\dabs{f}_{S;\,m}$ does not change if
$S$ is replaced by its closure in $U$.

 \subsection*{Acknowledgements}
The author would like to thank Paul Gauthier and Herwig Hauser for various suggestions to improve this paper.

\section{Preliminaries}\label{sec:prelims}

\noindent
In this section we collect various auxiliary results used later in the paper: some remarks on difference functions of convex functions,  
estimates concerning factorials, and an easy bound for a sum of square roots.

\subsection*{The difference function of a convex function}
In the next lemma we let $I$ be an interval (of any kind) in $\R$, and
we let  $\phi\colon I\to\R$ be a convex function, that is, 
$$\phi\big( (1-\lambda) s+ \lambda t\big)\le (1-\lambda) \phi(s)+ \lambda \phi(t)\quad\text{ for each~$\lambda\in[0,1]$ and $s,t\in I$.}$$
 We consider the function 
$$\Delta\colon D:=\big\{(s,t)\in I\times I: s\neq t\big\}\to\R,\qquad \Delta(s,t):=\frac{\phi(s)-\phi(t)}{s-t}.$$
Note that   $D$ and  $\Delta$ are symmetric.  
The following fact is well-known:

\begin{lemma}\label{lem:convex diff quot}
Let $s<t<u$ be in $I$; then $\Delta(s,t)\le  \Delta(s,u) \le \Delta(t,u)$.
\end{lemma}
\begin{proof}
We have $t=\lambda s+\mu u$ where $\lambda:=(u-t)/(u-s)$ and $\mu:=1-\lambda=(t-s)/(u-s)$. Then
$\phi(t) \le \lambda\phi(s)+\mu\phi(u)$  by convexity of $\phi$.
Subtracting $\phi(s)$ from both sides of this inequality yields
$$\phi(t)-\phi(s)\le \mu\big(\phi(s)-\phi(u)\big) = (t-s)\Delta(s,u),$$
and so we get $\Delta(s,t)\le  \Delta(s,u)$, whereas
subtracting $\phi(t)+\lambda\big(\phi(s)-\phi(u)\big)$ from both sides of this inequality   yields 
$$(u-t)\Delta(s,u)=\lambda\big(\phi(u)-\phi(s)\big) \le \phi(u)-\phi(t)$$
and so $\Delta(s,u) \le \Delta(t,u)$.
\end{proof}

\noindent
Suppose now that $I$ is not bounded from above.
Then by Lemma~\ref{lem:convex diff quot}, the function
$$t\mapsto \Delta \phi(t):=-\Delta(t,t+1)=\phi(t)-\phi(t+1)\colon I\to\R$$ is decreasing.
Also note that if $\phi\in\C^2(I)$,
then $\phi''\ge 0$ by convexity of $\phi$,   so~$\Delta\phi(t) \ge -\phi'(t+1)$  for $t\in I$, by the Mean Value Theorem.
This yields:

\begin{lemma}\label{lem:Deltaphi}
Suppose  $\phi\in\C^2(I)$, $\phi>0$, and $C\in\R^>$ are such that $ \phi'(t)/ \phi(t)  \le -C/t$ for each $t\in I$ with $t>0$.
Then $\Delta\phi(t) \ge C\phi(t+1)/(t+1)$ for $t\in I$, $t>0$.
\end{lemma}

\subsection*{Cheap approximations to the factorial}
The following estimates for $n!$ are not as tight as Stirling's approximation $n!\sim \sqrt{2\pi n}(n/{\ex})^n$ (as $n\to\infty$), but good enough for our purposes:

\begin{lemma}\label{lem:fact bd}
Suppose $n\ge 1$. Then
$$\ex\left(\frac{n}{\ex}\right)^n \le n!\le \frac{\ex^2}{4}\left(\frac{n+1}{\ex}\right)^{n+1}.$$
\end{lemma}
 \begin{proof}
For $t>1$, from the inequalities $t-1\le \lfloor t\rfloor\le t$ we obtain $\log(t-1)\le\log \lfloor t\rfloor\le \log t$. Hence by integration 
$$\int_2^{n+1} \log(t-1)\,dt \le \int_2^{n+1}\log\lfloor t\rfloor\,dt = \sum_{k=2}^n\log k \le \int_2^{n+1}\log t\,dt$$
and thus
$$n\big(\log(n)-1\big)+1 \le \log(n!) \le (n+1)\big(\log(n+1)-1\big)-\big(2\log(2)-2\big),$$
and this yields the claimed inequality.
 \end{proof}


\begin{cor}\label{cor:fact bd}
Let $\rho\in\R$ with $\rho\ge\ex t>0$. Then  
$$t^{\rho }\frac{4}{\ex^2}\left(\frac{\ex}{\rho +2}\right)^{\rho +2} \le
\frac{t^{n}}{n!}\qquad\text{for $n\le \rho$.}$$
\end{cor}
\begin{proof}
Put $m:=\lceil \rho\rceil$. Then
$0\le m-1<\rho\le m$ and thus 
\begin{equation}  \label{eq:fact bd, 1}
\frac{4}{\ex^2}\left(\frac{\ex}{\rho +2}\right)^{\rho +2} \le \frac{4}{\ex^2}\left(\frac{\ex}{m+1}\right)^{m+1}\le \frac{1}{m!},
\end{equation}
where the second inequality holds by Lemma~\ref{lem:fact bd}.
Since $m\ge\ex t$, by      Lemma~\ref{lem:fact bd} again:
$$\frac{t^m}{m!} \le t^m\frac{1}{\ex}\left(\frac{\ex}{m}\right)^m \le \left(\frac{\ex t}{m}\right)^m \le 1.$$
Moreover,  $t^n/n! \le t^{n+1}/(n+1)!$ for $n\le t-1$, and $t^n/n! \ge t^{n+1}/(n+1)!$  for~$n\ge t-1$. Hence
$\min\{ t^n/n! : n \le m \} = t^m/m!$, so in particular, 
\begin{equation}  \label{eq:fact bd, 2}
t^m/m!\le t^n/n!\quad \text{ for $n\le \rho$.}
\end{equation}
Now if $t\le 1$ then for $n\le\rho$ we have $t^\rho \le t^n$ and thus by \eqref{eq:fact bd, 1}:
$$ t^\rho \frac{4}{\ex^2}\left(\frac{\ex}{\rho +2}\right)^{\rho +2} \le \frac{t^n}{m!}\le \frac{t^n}{n!}.$$
If $t>1$, then $t^\rho\le t^m$ and so
$$ t^{\rho }\frac{4}{\ex^2}\left(\frac{\ex}{\rho +2}\right)^{\rho +2} \le \frac{t^m}{m!} \le \frac{t^n}{n!},$$
using both \eqref{eq:fact bd, 1} and  \eqref{eq:fact bd, 2}.
\end{proof}

\subsection*{Bounding  a sum of square roots}
For all $\alpha,\beta\in\R^{\ge}$ we have 
$$\sqrt{\alpha+\beta}\le\sqrt{\alpha}+\sqrt{\beta}\le\sqrt{2(\alpha+\beta)}.$$
This observation   easily yields the following  fact, recorded here for later use:

\begin{lemma}\label{lem:sqrt ineq}
If $a,b,c\in\R^{\ge}$, then $\sqrt{a+b}+\sqrt{c} \le \sqrt{a}+\sqrt{2(b+c)}$.
\end{lemma}

\section{Proof of Theorem~\ref{thm:WAP GK}}\label{sec:WAP GK}

\noindent
We begin with a general estimate coming out of a case of Fa\`a di Bruno's Formula~\cite[\S{}3.4, Theorem~B]{Comtet}. 
For this, let~$f\in\C^\infty(I)$ where $I$ is an open interval in $\R$   and~${g\in\C^r(J)}$ where $J$ is an open interval with~$f(I)\subseteq J$, and  set~$h:=g\circ f\in\C^r(I)$.
Let~$B_{mn}(y_1,\dots,y_{n-m+1})\in \Q[y_1,\dots,y_{n-m+1}]$ ($m\le n$) be the Bell polynomials
as defined in \cite[12.5]{ADH}. Then for $n\le r$ we have
\begin{equation}\label{eq:FdB}
h^{(n)} = \sum_{m=0}^n (g^{(m)}\circ f) \cdot B_{mn}(f',f'',\dots,f^{(n-m+1)}).
\end{equation}
See, e.g., \cite[\S{}3.3]{Comtet} for basic facts about the $B_{mn}$. For example, they  
have coefficients in $\N$, with $B_{00}=1$,
$B_{0n} = 0$  for $n \ge 1$, and 
 for $1\le m\le n$, $B_{mn}$ is homogeneous of degree $m$. Let $\stirlingsecond{n}{m}$ denote the Stirling numbers of the second kind,
that is, the number of equivalence relations  on an $n$-element set with exactly $m$ equivalence classes; cf.~\cite[p.~576]{ADH}.
 They obey the   recurrence relations
$$\stirlingsecond{n+1}{m}=m\stirlingsecond{n}{m}+\stirlingsecond{n}{m-1}$$
with side conditions $\stirlingsecond{0}{0}=1$ and $\stirlingsecond{n}{0}=\stirlingsecond{0}{m}=0$ for $m,n\ge 1$.
The numbers
$$B_n := \sum_{m=0}^n \stirlingsecond{n}{m}$$
are known as the Bell numbers. We have $B_n\le n^n$ for $n\ge 1$.
To see this note that   $B_n$ is the number of equivalence relations on $[n]:=\{1,\dots,n\}$, 
and there is a surjection from the set $[n]^{[n]}$ of all maps $\lambda\colon [n]\to [n]$
to the collection of  equivalence relations on~$[n]$, which maps $\lambda$ to the  equivalence relation $\sim_\lambda$ on  $[n]$ 
given by~${i\sim_\lambda j} :\Leftrightarrow {\lambda(i)=\lambda(j)}$, for $i,j\in [n]$. 
(Much more is known about the asymptotics of $B_n$, see, e.g., \cite[p.~108]{deBruijn},
but we  won't need this here.)  
We have~$B_{mn}(1,1,\dots,1)=\stirlingsecond{n}{m}$; cf.~\cite[\S{}3.3, Theorem~B]{Comtet}. 
These observations   yield:

\begin{lemma}\label{lem:comp bd}
Suppose $1\le n\le r$ and $t\in I$, and let $F,G\in\R$, $F\ge 1$  be such that 
$$ \abs{f^{(k)}(t)}\le F\quad\text{ for $k=1,\dots,n$,} \qquad \abs{g^{(m)}(f(t))}\le G\quad\text{ for $m=1,\dots,n$.}$$ Then
$\abs{h^{(n)}(t)} \le G\cdot (n F)^n$.
\end{lemma}
\begin{proof}
By \eqref{eq:FdB} we  have 
\begin{align*}
\abs{h^{(n)}(t)} & \le  \sum_{m=1}^n   G\cdot \abs{ B_{mn}(f'(t),f''(t),\dots,f^{(n-m+1)}(t)) } \\
& \le  \sum_{m=1}^n G\cdot B_{mn}(1,1,\dots,1)\cdot F^m \le  G\cdot B_n\cdot F^n \le G\cdot (n F)^n.\qedhere
\end{align*}
\end{proof}

\noindent
We now prove Theorem~\ref{thm:WAP GK}. Thus let~$I$ be an
open interval,  $f\in\C^r(I)$,
and~${\varepsilon,\rho\in\C(I)}$ where $\varepsilon>0$, $r\ge\rho\ge 0$. 
We first assume that $I$ is bounded, say~$I=(a,b)$ where~$a<b$, and
we need to show the existence of a holomorphic function~$g\colon\Co\setminus\{a,b\}\to\Co$
such that~$g(I)\subseteq\R$ and~$\abs{({f-g})^{(n)}(t)}<\varepsilon(t)$ for all~$t\in I$ and~$n\le\rho(t)$. It is straightforward to arrange that~$I=(-1,1)$:
Set~$\alpha:=\frac{b-a}{2}$, $\beta:=\min\{\alpha,1\}$, and $J:=(-1,1)$, and
consider the holomorphic bijection~$\phi\colon   \Co \to  \Co$,
$\phi(z):= \alpha\cdot z+\textstyle\frac{1}{2}(a+b)$,
with compositional inverse $\phi^{\operatorname{inv}}$.
Put~$f_0:=(f\circ\phi)|_J\colon J\to\R$ and let~$\varepsilon_0,\rho_0\in\C(J)$ be given by 
$\varepsilon_0(s):= \beta^{\rho(\phi(s))}\varepsilon(\phi(s))$ and~$\rho_0(s):=\rho(\phi(s))$, respectively.
Suppose we have a holomorphic function~$g_0\colon\Co\setminus\{\pm 1\}\to\Co$  such that~$g_0(s)\in\R$ and
$$\abs{(f_0-g_0)^{(n)}(s)}< \varepsilon_0(s)\quad\text{for   $s\in J$, $n\le\rho_0(s)$.}$$ 
Consider the holomorphic function $g:=g_0\circ\phi^{\operatorname{inv}}\colon \Co\setminus\{a,b\}\to\Co$.   Then for~$t\in I$ and~$n\le\rho(t)$,
setting   $s:=\phi^{\operatorname{inv}}(t)\in J$ we have $g(t)=g_0(s)\in\R$ and
$$(f_0-g_0)^{(n)}(s)=\big((f-g)\circ \phi\big)^{(n)}(s)=(f-g)^{(n)}(t)\cdot \alpha^n$$
and so
$$\abs{(f-g)^{(n)}(t)} < \beta^{\rho(t)}\alpha^{-n}\varepsilon(t) \le \varepsilon(t).$$
Hence replacing $f$, $\varepsilon$, $\rho$ by $f_0$, $\varepsilon_0$, $\rho_0$, respectively, we may arrange   $I=(-1,1)$, which we assume from now on.
The rational function
$$\Phi\colon\Co\setminus\{\pm 1\} \to \Co,\qquad \Phi(z):=\frac{2z}{1-z^2}=\frac{1}{1-z}-\frac{1}{z+1}$$
restricts to an analytic bijection
$\phi\colon I\to\R$. Let~${\phi^{\operatorname{inv}}\colon \R\to I}$ be its  (analytic) compositional in\-verse.
An easy induction on $n$ shows that for each $n$ and $t\in I$,
$$\phi^{(n)}(t) =  n! \left( \frac{1}{(1-t)^{n+1}} + (-1)^{n+1} \frac{1}{(t+1)^{n+1}}\right)$$ 
and thus
\begin{equation}\label{eq:bd psin}
\abs{\phi^{(n)}(t)} \leq n! \frac{(t+1)^{n+1}+(1-t)^{n+1}}{(1-t^2)^{n+1}} \leq \frac{n!\cdot 2^{n+1}}{(1-t^2)^{n+1}}.
\end{equation}
Consider the function $\varepsilon_*\in\C(I)$ given by
$$\varepsilon_*(t):= \varepsilon(t) \cdot \left( \frac{4}{(\rho(t)+1)\ex^2}\cdot \left( \frac{\ex(1-t^2)}{2(\rho(t)+1)} \right)^{\rho(t)+1} \right)^{\rho(t)}  \quad\text{for $t\in I$.}$$
Whitney's Theorem  (in the case $I=\R$) applied to  $f_*:=f\circ\phi^{\operatorname{inv}}\in\C^r(\R)$ as well as~${\varepsilon_*\circ \phi^{\operatorname{inv}}},\rho\circ \phi^{\operatorname{inv}}\in\C(\R)$ in place of
$f$, $\varepsilon$, $\rho$, respectively,  yields an entire function~$g_*\colon\Co\to\Co$ such that $g_*(\R)\subseteq\R$ and, with~$h_*:=f_*-g_*|_{\R}$:
\begin{equation}\label{eq:h*n}
\abs{h_*^{(n)}(s)} < \varepsilon_*(\phi^{\operatorname{inv}}(s))\qquad\text{ whenever $n\leq\rho(\phi^{\operatorname{inv}}(s))$.}
\end{equation}
Consider the  holomorphic function $g:=g_*\circ\Phi\colon\Co\setminus\{\pm 1\}\to\Co$  and put~$h:=f-g|_I$, so $h_*=h\circ\phi$. Then for $t\in I$ 
we have $\abs{h(t)}<\varepsilon(t)$, and
if~$1\le n\leq \rho(t)$, then
\begin{align*}
\abs{h^{(n)}(t)} &< \varepsilon_*(t)\cdot \left(\frac{n\cdot n!\cdot 2^{n+1}}{(1-t^2)^{n+1}}\right)^n \\
&\le \varepsilon_*(t)\cdot \left( \frac{n\ex^2}{4}\cdot \left(\frac{2(n+1)}{\ex(1-t^2)} \right)^{n+1} \right)^n \\
&\le  \varepsilon_*(t)\cdot \left( \frac{\big(\rho(t)+1\big)\ex^2}{4}\cdot \left(\frac{2(\rho(t)+1)}{\ex(1-t^2)} \right)^{\rho(t)+1} \right)^{\rho(t)} = \varepsilon(t),
\end{align*}
using 
\eqref{eq:bd psin} and \eqref{eq:h*n} in combination with Lemma~\ref{lem:comp bd} for the first inequality and
Lemma~\ref{lem:fact bd} for the  second inequality.

This shows Theorem~\ref{thm:WAP GK} in the case of a bounded interval. Now suppose $I\subseteq\R$ is   unbounded, say $I=(a,\infty)$.  Replacing $a$, $f$, $\varepsilon$, $\rho$ by $0$, $t\mapsto f(t+a)$, $t\mapsto\varepsilon(t+a)$, and~$t\mapsto \rho(t+a)$, respectively, we first arrange $a=0$, so $I=\R^>$. 
Now consider   the rational function
$$\Psi\colon\Co\setminus\{0\} \to \Co,\qquad \Psi(z):=z-\frac{1}{z},$$
which restricts to an analytic bijection $\psi\colon \R^>\to\R$ with (analytic) compositional inverse $\psi^{\operatorname{inv}}\colon\R\to\R^>$.
For $t>0$ we have
$$\psi'(t)=1+\frac{1}{t^2},\qquad \psi^{(n)}(t)= \frac{n!}{(-t)^{n+1}}\quad\text{if $n > 1$.}$$
For $t>0$, putting
$$\rho_+(t):=\rho(t)+\ex t,\qquad\beta(t):=1+t^{-(\rho_+(t)+1)} \frac{\ex^2}{4} \left( \frac{\rho_+(t)+2}{\ex} \right) ^{\rho_+(t)+2}$$
we obtain $\abs{\psi^{(n)}(t)} \le \beta(t)$ for $1\le n\le \rho_+(t)$, by Corollary~\ref{cor:fact bd}.
Consider 
$$\varepsilon_*\in\C(\R^>),\qquad \varepsilon_*(t):= \varepsilon(t) \cdot \big( \rho(t) \beta(t) \big)^{-\rho(t)}  \quad\text{for $t>0$.}$$
As in the previous case,  Whitney's Theorem  applied to  $f_*:=f\circ\psi^{\operatorname{inv}}\in\C^r(\R)$ as well as~${\varepsilon_*\circ \psi^{\operatorname{inv}}},\rho\circ \psi^{\operatorname{inv}}\in\C(\R)$ in place of
$f$, $\varepsilon$, $\rho$, respectively,  yields an entire function~$g_*\colon\Co\to\Co$ such that $g_*(\R)\subseteq\R$ and, with~$h_*:=f_*-g_*|_{\R}$:
$$\abs{h_*^{(n)}(s)} < \varepsilon_*(\psi^{\operatorname{inv}}(s))\qquad\text{ whenever $n\leq\rho(\psi^{\operatorname{inv}}(s))$.}$$
Let $g:=g_*\circ\Psi$, a holomorphic function   $\Co\setminus\{0\}\to\Co$, and $h:=f-g|_{\R^>}\in\C^r(\R^>)$, so $h_*=h\circ\psi$. Then  for $t>0$ we   have
$\abs{h(t)}<\varepsilon(t)$, and  if $1\le n\le \rho(t)$, then 
$$\abs{h^{(n)}(t)} 	 < \varepsilon_*(t)\cdot\big( n \beta(t) \big)^n 
					 \le \varepsilon_*(t)\cdot \big( \rho(t) \beta(t) \big)^{\rho(t)} = \varepsilon(t)$$
as required. This finishes the proof of Theorem~\ref{thm:WAP GK}. \qed

\medskip
\noindent
Next we prove Corollary~\ref{cor:WAP GK}. Let $I$, $f$, $\varepsilon$, $\rho$ be as in the corollary. The case where~$I$ is open is covered by Theorem~\ref{thm:WAP GK}, and the case where $I$ is closed and bounded is a consequence of the Weierstrass Approximation Theorem \cite[p.~33]{Nara}.  Suppose~${I=[\alpha,\beta)}$ where $-\infty<\alpha<\beta\le\infty$. 
(The remaining case where $I=(\alpha,\beta]$  with $-\infty\le\alpha<\beta<\infty$ is treated in a similar way.) 
Then $\operatorname{fr}(I)=\{\beta\}$. We set~$I_1:=(-\infty,\beta)$
and now use a well-known fact:

\claim{There is a $\C^r$-function $f_1\colon I_1\to\R$ which extends $f$.}

\noindent
To see this take $\delta\in\R^>$ and an extension of $f$ to a $\C^r$-function $(\alpha-\delta,\beta)\to\R$, also denoted by $f$.
Let $\theta\in\C^\infty(\R)$ be such that $\theta(t)=0$ for $t\le \alpha-(\delta/2)$ and $\theta(t)=1$ for $t\ge\alpha$.
(See, e.g., Section~\ref{sec:bump fns} below.)
Then $f_1\colon I_1\to\R$ given by~$f_1(t):=f(t)\theta(t)$ if~$t\in (\alpha-\delta,\beta)$ and $f_1(t):=0$ if $t\le\alpha-\delta$
 is $\C^r$ and extends $f|_{I}$.

To finish the proof, take $f_1$ as in the claim and extend $\varepsilon$, $\rho$ to $\varepsilon_1,\rho_1\in\C(I_1)$ by~$\varepsilon_1(t):=\varepsilon(\alpha)$,
$\rho_1(t):=\rho(\alpha)$ for $t\le\alpha$; then Theorem~\ref{thm:WAP GK} applied to~$f_1$,~$\varepsilon_1$,~$\rho_1$ in place of $f$, $\varepsilon$, $\rho$
yields $g_1\in\C^\omega(I_1)$ which extends to a holomorphic function~${\Co\setminus\{\beta\}\to\Co}$ such that $\abs{(f_1-g_1)^{(n)}(t)}<\varepsilon_1(t)$ for each $t\in I_1$ and $n\le\rho_1(t)$. Then $g:=g_1|_I\in\C^\omega(I)$ has an extension to a holomorphic function~${\Co\setminus\{\beta\}\to\Co}$
and satisfies $\abs{(f-g)^{(n)}(t)}<\varepsilon(t)$ for $t\in I$ and $n\le\rho(t)$ as required.
\qed
 
\section{Bounds on the Derivatives of Bump Functions}\label{sec:bump fns}

\noindent
By a {\it bump function}\/ we mean here a $\mathcal C^\infty$-function $\alpha\colon\R\to\R$
which is $0$ on~$(-\infty,0]$,   strictly increasing on $[0,1]$, and $1$ on $[1,+\infty)$. Such bump functions play an important role in many constructions in analysis, e.g.,
partitions of unity. 
In this section we give   a   construction of  a bump function $\alpha$ with controlled growth of its derivatives $\alpha^{(n)}$. 
As an auxiliary result we first establish the following formula for the derivatives of reciprocals of nonzero elements of   a differential field $K$:

\begin{prop}\label{prop:ders of recip}
For    $\phi\in K^\times$ and $n\ge 1$ we have
\begin{equation}\label{eq:deriv of recip}
(\phi^{-1})^{(n)} = \sum_{k=1}^n (-1)^k {n+1 \choose k+1} (\phi^{-1})^{k+1} (\phi^k)^{(n)}.
\end{equation}
\end{prop}

\begin{proof}
We work in the setting of \cite[Chapter~12]{ADH} and assume familiarity with the concepts and the notations introduced there.
Let $\phi$ range over $K^\times$ and~$i$,~$j$,~$k$,~$l$ over~$\N$. Put
$$h_\phi:=\sum_n \left(\frac{\phi^{(n)}}{\phi}\right)\frac{z^n}{n!} \in 1+zK[[z]] \subseteq K[[z]]^\times.$$
By the generalized  Leibniz Rule,  $\phi\mapsto h_\phi\colon K^\times\to K[[z]]^\times$ is a group morphism, so  we obtain a group morphism $\phi\mapsto [h_\phi]\colon K^\times\to \mathfrak{tr}_K$ whose image is contained in the Appell group $\mathcal A$ over $K$
(cf.~\cite[p.~565]{ADH}). 
 For $i\leq j$ we have~$[h_\phi]_{ij}={j\choose i}\phi^{(j-i)}/\phi$.
Since~$\mathcal A\subseteq 1+\mathfrak{tr}_K^1$,  we have $\big([h_\phi]-1\big)^k \in  \mathfrak{tr}_K^k$ for~$k\ge 1$, and
$$[h_{\phi^{-1}}]=[h_\phi]^{-1} = \sum_k (-1)^k \big([h_\phi]-1\big)^k = \sum_k (-1)^k\left(\sum_{l=0}^k {k\choose l}(-1)^{k-l} [h_{\phi^l}]\right) .$$
For $n\ge 1$ this yields
\begin{align*}
(\phi^{-1})^{(n)}\phi = [h_{\phi^{-1}}]_{0n} 	&= \sum_{k=1}^n (-1)^k\left( \big([h_\phi]-1\big)^k \right)_{0n} \\
										&= \sum_{k=1}^n (-1)^k\left( \sum_{l=1}^k {k\choose l} (-1)^{k-l} [h_{\phi^l}]_{0n} \right) \\
										&= \sum_{l=1}^n (-1)^l \left(\sum_{k=l}^n {k\choose l}\right) (\phi^{-1})^{l} (\phi^l)^{(n)} \\
										&= \sum_{l=1}^n (-1)^l {n+1 \choose l+1} (\phi^{-1})^{l} (\phi^l)^{(n)}
\end{align*}
where for the last equality we used the well-known identity
$$\sum_{k=l}^n {k\choose l} = {n+1 \choose l+1},$$
which has an easy proof by induction on $n$.
\end{proof}

\begin{remark}
The previous proposition also holds if $K$ is any differential {\it ring}\/ (in the sense of \cite[4.1]{ADH}). To see this let $Y$ 
be a differential indeterminate over~$\Q$. Then the identity~\eqref{eq:deriv of recip} holds for $\Q\<Y\>$, $Y$ in place of $K$, $\phi$, respectively. It now suffices to note that given any differential ring~$K$ and a unit $\phi\in K^\times$
we have the morphism~$S^{-1}\Q\{Y\}\to K$ of differential rings with~$Y\mapsto \phi$, where 
$$S^{-1}\Q\{Y\}=\big\{Y^{-n}P:P\in \Q\{Y\}, n\ge 0\big\}\subseteq \Q\<Y\>$$ is the localization
of $\Q\{Y\}$ at its multiplicative subset $S:=\{1,Y,Y^2,\dots\}$.
\end{remark}

\noindent
In the rest of this section we prove:

\begin{prop}\label{prop:bump fn}
There are a bump function $\alpha$ and   constants $c,d\in\R^>$ such that~$\dabs{\alpha}_n \leq c\,n^{dn}$ for each $n\ge 1$.
\end{prop}

\noindent
We begin our construction by studying the function  $\theta\colon\R\to\R^{\ge}$ given by
$\theta(t):=\ex^{-1/t}$ if $t>0$ and $\theta(t):=0$ otherwise.
For each $n$ and $t>0$  we have
$$\theta^{(n)} = \frac{p_n(t)}{t^{2n}}\theta(t),$$
where $p_n\in\R[T]$ is the polynomial given recursively by $p_0=1$ and
\begin{equation}\label{eq:pn+1}
p_{n+1} = T^2p_n' - (2nT-1)p_n.
\end{equation}
Thus $$p_1=1,\quad p_2=-2T+1,\quad p_3=6T^2-6T+1,\quad p_4 = -24T^3+36T^2-12T+1,\quad\dots.$$ 
In general, $\deg p_n=n-1$ for~$n\ge 1$. Since for each $m$ we have
$t^{-m}\ex^{-1/t}\to 0$ as~$t\to 0^+$, this yields that
$\theta^{(n)}(t)\to 0$ as $t\to 0^+$, so
 $\theta$ is $\mathcal C^\infty$. We now want to bound the quantities $\dabs{\theta^{(n)}}_{[0,1]}$.
First, some notation and a lemma.

\begin{notation}
For a polynomial~$p=a_0+a_1T+\cdots+a_nT^n\in \R[T]$ ($a_0,\dots,a_n\in\R$) we   let~$|p|:=\max\big\{|a_0|,\dots,|a_n|\big\}$.
\end{notation}

\begin{lemma}
$\abs{p_n}\leq 2^{n-1}n!$ for $n\ge 1$.
\end{lemma}
\begin{proof}
This is clear for $n=1$. Suppose we have shown the inequality for some~${n\geq 1}$.
Let $a_0,\dots,a_{n-1},b_0,\dots,b_n\in\R$ with $p_n=a_0+a_1T+\cdots+a_{n-1}T^{n-1}$ and $p_{n+1}=b_0+b_1T+\cdots+b_nT^n$. Then 
\begin{align*}
T^2p_n' &\ = \ \sum_{m=1}^n (m-1)a_{m-1}T^m,\\  (2nT-1)p_n &\ =\ -a_0 + \sum_{m=1}^{n-1} (2n a_{m-1}-a_m) T^m + 2n a_n T^n.
\end{align*}
Hence by the recursion relation \eqref{eq:pn+1}  we have
$$b_0 = -a_0,\quad  
 b_m=a_m+(-2n+m-1)a_{m-1}\quad\text{for $m=1,\dots,n-1$,} \quad b_n = -2n a_{n-1}$$
and this yields $\abs{p_{n+1}}\leq 2^n(n+1)!$ as required.
\end{proof}

\noindent Next, given $\lambda\in\R^>$, consider    the function $t\mapsto \zeta(t):=t^{-\lambda}\ex^{-1/t}\colon\R^>\to\R^>$.
Then
$$\zeta'(t)/\zeta(t)=-\lambda/t+1/t^2=(-\lambda+1/t)/t,$$ so $\max\zeta(\R^>)=\zeta(1/\lambda)=\lambda^\lambda\ex^{-\lambda}$. Since the function~$\lambda\mapsto \lambda^\lambda\ex^{-\lambda}\colon\R^{\ge 1}\to\R^>$ is strictly increasing,
this yields, for~$t\in [0,1]$ and~$n\ge 1$:
$$\abs{\theta^{(n)}(t)}=\left|\frac{p_n(t)}{t^{2n}}\ex^{-1/t}\right| \leq n\cdot  \abs{p_n}\cdot (n+1)^{n+1} \ex^{-(n+1)}$$
and thus
$$\log\, \abs{\theta^{(n)}(t)} \le \log n+(n-1)\log 2+n\log n+(n+1)\big(\log(n+1)-1\big) \leq 3n\log n.$$
Therefore 
\begin{equation}\label{eq:derivatives of f}
\dabs{\theta^{(n)}}_{[0,1]}\leq n^{3n}\qquad\text{for   $n\ge 1$.}
\end{equation}
Let $\mu\in\R^>$. Then $\theta(t)^\mu=\theta(t/\mu)$ and thus $(\theta^\mu)^{(n)}(t) = \theta^{(n)}(t/\mu)/\mu^n$, hence 
$$\dabs{ (\theta^\mu)^{(n)} }_{[0,1]} = 
\mu^{-n} \dabs{\theta^{(n)}}_{[0,1/\mu]} \le  \dabs{\theta^{(n)}}_{[0,1]}\quad\text{if $\mu\ge 1$.}$$ Also using \eqref{eq:derivatives of f}, this yields:
\begin{equation}\label{eq:derivatives of f^r}
\dabs{ \theta^\mu  }_{[0,1];\,n}   \le n^{3n} \quad\text{if   $\mu,n\ge 1$.}
\end{equation}
We note that \eqref{eq:derivatives of f^r} also holds with $\theta$
replaced by the $\C^\infty$-function $\theta_*\colon\R\to\R^\ge$ given by
 $\theta_*(t):=\theta(1-t)$.
Since $\theta+\theta_*>0$,   we obtain the $\mathcal C^\infty$-function
$$\alpha:= \theta/(\theta+\theta_*)\colon\R\to\R,$$
and this is a bump function: we have $\alpha(t)=0$ for $t\le 0$ and  $\alpha(t)=1$ for~$t\ge 1$, and since 
for $t\in (0,1)$ we also have
$$\alpha'(t) = \frac{\theta'(t)\theta(1-t)+\theta(t)\theta'(1-t)}{\big(\theta(t)+\theta(1-t)\big)^2} = 
\theta(t)\theta(1-t)  \frac{ t^{-2} + (1-t)^{-2}  }{ \big(\theta(t)+\theta(1-t)\big)^2 } >0,$$
the restriction of $\alpha$ to $[0,1]$ is strictly increasing. Our goal is  now to compute bounds on $\dabs{\alpha}_n$, using~\eqref{eq:derivatives of f^r}. For this we  apply the remark after Proposition~\ref{prop:ders of recip} to the differential ring~$K=\mathcal C^\infty(\R)$ and the unit~$\phi:=\theta+\theta_*$ of $K$.
Suppose~$n\ge 1$. We first note that  
$$\phi^k = \sum_{l=0}^k {k\choose l} \theta^{k-l} \theta_*^l$$
where by \eqref{eq:prod norm} and  \eqref{eq:derivatives of f^r} we have
$$\dabs{  \theta^{k-l}  \theta_*^l }_{[0,1];\,n} \le 2^n n^{6n}\quad\text{for   $l=0,\dots,k$ and $m=0,\dots,n$,}$$
hence
$\dabs{\phi^k}_{[0,1];\,n} \leq 2^{k+n} n^{6n}$.
Next we note that for $t\in(0,1)$ we have
$$\phi'(t) = \frac{\theta(t)}{t^2}-\frac{\theta_*(t)}{(1-t)^2},$$
hence $\phi'(t)<0$ if $t<1/2$ and $\phi'(t)>0$ if $t>1/2$. This yields $\phi\ge \phi(1/2)=\ex^{-2}$ and so
$\dabs{ (\phi^{-1})^{k+1} }_{[0,1]} \leq \ex^{2(k+1)} \le 2^{3(k+1)}$. Using \eqref{eq:deriv of recip} we thus conclude:
\begin{align*}
\dabs{ (\phi^{-1})^{(n)} }_{[0,1]} &\le 
\sum_{k=1}^n {n+1 \choose k+1} \cdot \dabs{(\phi^{-1})^{k+1}}_{[0,1]} \cdot\dabs{ (\phi^k)^{(n)}}_{[0,1]} \\
&\le \sum_{k=1}^n {n+1 \choose k+1}\cdot 2^{3(k+1)}\cdot 2^{k+n}n^{6n}\leq 2^{6n+4}n^{6n}. 
\end{align*}
Now $\alpha=\theta\cdot \phi^{-1}$ and thus
\begin{align*} \dabs{\alpha}_n=\dabs{\alpha}_{[0,1];\,n} &\le 2^n \cdot \dabs{\theta}_{[0,1];\,n} \cdot \dabs{ \phi^{-1} }_{[0,1];\,n}\\
&\le 2^n \cdot n^{3n}\cdot 2^{6n+4} n^{6n} 
= 2^{7n+4}n^{9n} \le cn^{dn}
\end{align*}
where $c:=2^{11}$, $d:=16$.
This concludes the proof of Proposition~\ref{prop:bump fn}. \qed

\begin{remark}
Note that the bound in this proposition is qualitatively optimal in the following sense: there
is no  function $\gamma\colon\N\to\R^>$ with $\gamma(n)=o(n)$   such that~$\dabs{\alpha^{(n)}} =O(n^{\gamma(n)})$.  To see this suppose~$\beta$ is any bump
function.
Let~$n$ be given and put~$M:=\dabs{\beta^{(n)}}$. If $n\geq 1$ then~$\abs{\beta^{(n-1)}(t)} \leq Mt$ for $t\in (0,1)$,
hence if $n\ge 2$ then~$\abs{\beta^{(n-2)}(t)} \leq Mt^2/2$ for~$t\in (0,1)$, etc., thus~$\abs{\beta(t)} \leq Mt^n/n!$ for $t\in (0,1)$. Hence there is no~$c\in\R^>$   such that $\dabs{\beta^{(n)}}\leq c^n n!$
for all $n$, since for such $c$ we would have $\beta(t)=0$ for each~$t\in [0,1/c]$, contradicting that $\beta|_{[0,1]}$ is
strictly increasing. Now the claim follows from this observation and Lemma~\ref{lem:fact bd}.
\end{remark}

\noindent
Let $\alpha$, $c$, $d$ be as in Proposition~\ref{prop:bump fn}, where
we may assume $c\ge 1$. With $0^0:=1$,
put   $C_n:=cn^{dn}$,
so $1 = C_0\leq C_1\leq\cdots$. For $a<b$ in $\R$,  we define the increasing $\Cc^{\infty}$-function 
$\alpha_{a,b}\colon \R \to \R$ by
\begin{equation}\label{eq:alphaab}
\alpha_{a,b}(t)\ :=\  \alpha\!\left(\frac{t-a}{b-a}\right),
\end{equation}
so $\alpha_{a,b}(t)=0$ for $t\le a$ and $\alpha_{a,b}(t)=1$ for $t\ge b$. Also, 
$$\big|\alpha_{a,b}^{(m)}(t)\big|\ \le\  \frac{C_m}{(b-a)^m}\ 
\text{ for all $m$ and $t$}.$$
The case most relevant  for us later is when $b-a\le 1$; then we have
\begin{equation}\label{eq:alphaab(m)}
\dabs{\alpha_{a,b}}_n \le \frac{C_n}{(b-a)^n}.
\end{equation}
Next let $a<b<a_*<b_*$ be in $\R$, and with $\varepsilon:=\frac{1}{3}(b-a),\varepsilon_*:=\frac{1}{3}(b_*-a_*)\in\R^>$ define 
the (hump) function $\alpha_{a,b,a_*,b_*}\colon\R\to\R$ by
\begin{equation}\label{eq:hump fn}
\alpha_{a,b,a_*,b_*}(t)\ := \ \begin{cases}
\alpha_{a+\varepsilon,b-\varepsilon}(t) 	& \text{if $t\le b$,} \\
1-\alpha_{a_*+\varepsilon_*,b_*-\varepsilon_*}(t)	& \text{otherwise.}
\end{cases}
\end{equation}
Then $\alpha_{a,b,a_*,b_*}(t)=0$ if $t\notin [a,b_*]$, $\alpha_{a,b,a_*,b_*}(t)=1$ if $t\in [b,a_*]$, 
and $\alpha_{a,b,a_*,b_*}$ is    increasing on $[a,b]$ and   decreasing on $[a_*,b_*]$.
Suppose   $\varepsilon, \varepsilon_*\le 1$. Then   by \eqref{eq:alphaab(m)},
\begin{equation}\label{eq:hump fn, bd}
\dabs{ \alpha_{a,b,a_*,b_*} }_n \ \le \ 3^n C_n\max\left\{ \frac{1}{(b-a)^n}, \frac{1}{(b_*-a_*)^n} \right\}.
\end{equation}
In particular, for suitable reals $c_1,d_1\geq 1$ (depending only on~$b-a$,~$b_*-a_*$,~$c$,~$d$)
we have  
$\dabs{ \alpha_{a,b,a_*,b_*} }_n  \le c_1 n^{d_1n}$  for each $n$.

\section{Weierstrass Transforms}\label{sec:Weierstrass transforms}

\noindent
Let $f\in\Cc^m(\R)$ be such that~$\supp f$ (=~the closure in
$\R$ of the set~$\R\setminus f^{-1}(0)$) is bounded; let also $\lambda$ range over $\R^>$.
The function~${f_\lambda\colon\R\to\R}$ given by
\begin{equation}\label{eq:flambda}
f_\lambda(t) := (\lambda/\pi)^{1/2}\int_{-\infty}^\infty f(s)\ex^{-\lambda (s-t)^2}\,ds
\end{equation}
is known as the (generalized) {\it Weierstrass transform}\/ with parameter $\lambda$; below we sometimes denote $f_\lambda$ also by $W_\lambda(f)$.
We could have replaced   the bounds~$-\infty$,~$\infty$ in this integral by any~$a$,~$b$    such that $\supp(f)\subseteq [a,b]$. A change of variables gives
$$f_\lambda(t)\  =\  (\lambda/\pi)^{1/2}\int_{-\infty}^\infty f(t-s)\ex^{-\lambda s^2}\,ds.$$
From the Gaussian integral~$\int_{-\infty}^\infty \ex^{-s^2}\,ds=\pi^{1/2}$ 
we get~$(\lambda/\pi)^{1/2}\int_{-\infty}^\infty \ex^{-\lambda s^2}\,ds=1$, hence
$\dabs{f_\lambda} \leq \dabs{f}$.
Also, $f_\lambda$  extends to an entire function; in particular, $f_\lambda\in\C^\omega(\R)$.
For $k\leq m$ we have $(f_\lambda)^{(k)}=(f^{(k)})_\lambda$, so $\dabs{f_\lambda}_m \leq \dabs{f}_m$.
See~\cite[Appendix~A]{fgh-an} for   proofs of these facts. (Entire functions which arise as Weierstrass transforms have been studied extensively \cite{Ditzian, Nessel, Pollard, Widder}.) Next we record an explicit version of \cite[Lemma~A.2]{fgh-an}, for which
we let $f\in\C^{m+1}(\R)$ have bounded    support.

\begin{lemma}\label{lem:bound lambda}
Let $\varepsilon\in\R^>$, and set 
$$\lambda(\varepsilon):=8(\dabs{f}_{m+1}\varepsilon^{-1})^2\log^+\!\big(2\sqrt{2}\,\dabs{f}_{m}\,\varepsilon^{-1}\big)\in\R^{\ge}.$$ 
Then for each $\lambda > \lambda(\varepsilon)$
we have $\dabs{f_\lambda-f}_m\leq\varepsilon$.
\textup{(}In particular, $\dabs{f_\lambda-f}_m\leq\varepsilon$ provided~$\lambda\ge 16\sqrt{2}(\dabs{f}_{m+1}\varepsilon^{-1})^3+1$.\textup{)}
\end{lemma}
\begin{proof}
Put $\delta:= (\varepsilon/2) / \dabs{f}_{m+1}$. By the Mean Value Theorem we have
$$\abs{ f^{(k)}(s)-f^{(k)}(t) } \leq \varepsilon/2\quad\text{whenever $\abs{s-t}\leq\delta$ and $k\leq m$.}$$
An easy computation shows that $\lambda > \lambda(\varepsilon)$ implies
$\sqrt{2}\dabs{f}_m\ex^{-(\lambda/2)\delta^2}\leq \varepsilon/2$, and by the argument in the proof of Lemma~A.2, this   guarantees  $\dabs{f_\lambda-f}\leq\varepsilon$.
\end{proof}

\section{Whitney Approximation with Bounds}\label{sec:WAP with bds}

\noindent
Let $(a_n)$, $(b_n)$, $(\varepsilon_n)$ be sequences in $\R$ and $(r_n)$
in~$\N$ such that
\begin{enumerate}
\item[(i)] $a_0=b_0$, $(a_n)$ is strictly decreasing, $(b_n)$ is strictly increasing, 
\item[(ii)] $\varepsilon_n>\varepsilon_{n+1}$ with $\varepsilon_n\to 0$ as $n\to\infty$,  and
\item[(iii)] $r_n\le r_{n+1}\le r$ for each $n$. 
\end{enumerate}
Set
$$K_n\ :=\ [a_n, b_n],\qquad L_n\ :=\ K_{n+1}\setminus K_n\ =\ [a_{n+1},a_n)\cup(b_n,b_{n+1}].$$
Then    $I:=\bigcup_n K_n$   is an open interval in $\R$. Let  $f\in \Cc^r(I)$. 
The proof of Whitney's Approximation Theorem given
in \cite[Appendix~A]{fgh-an} then produces a~$g\in\Cc^\omega(I)$ such that~$\dabs{{f-g}}_{L_n;\,r_n}<\varepsilon_n$ for each~$n$.
We recall the main lines of this argument.
First, we introduce some hump functions $\varphi_n$ as follows:
for $n\ge 1$, employing the $\C^\infty$-functions introduced in  \eqref{eq:hump fn}, we set
$$\alpha_n := \alpha_{a_{n+2},a_{n+1},a_n,a_{n-1}},\qquad 
\beta_n := \alpha_{b_{n-1},b_{n},b_{n+1},b_{n+2}}, \qquad \varphi_n:=\alpha_n+\beta_n,
$$
and also set $\varphi_0:=\alpha_{a_2,a_1,b_1,b_2}\in\Cc^\infty(\R)$. Then for each $n$ we have $\varphi_n=0$ on a neighborhood of~$K_{n-1}$ (satisfied automatically for $n=0$, by convention),  
${\varphi_n=1}$ on a neighborhood of~$\operatorname{cl}(L_n)=[a_{n+1},a_n]\cup[b_n,b_{n+1}]$, and~$\supp\varphi_n\subseteq K_{n+2}$.
 With~$M_n:=1+2^{r_n}\dabs{\varphi_n}_{r_n}$,   choose~$\delta_n\in\R^>$ so that for all $n$,
\begin{equation}\label{eq:WAP, 0}
2\delta_{n+1} \leq \delta_n,\qquad \sum_{m = n}^\infty \delta_m M_{m+1}\leq \varepsilon_n/4.
\end{equation}
We   inductively define sequences $(\lambda_n)$ in $\R^>$ and $(g_n)$ in~$\Cc^\omega(\R)$
 as follows:
 Let~$\lambda_m\in \R^{>}$ and $g_m\in \Cc^{\omega}(\R)$ for~$m<n$; then
consider  $h_n\in\Cc^r(\R)$ given by
$$h_n(t)\ :=\ \begin{cases} \varphi_n(t)\cdot \big(f(t)-\big(g_0(t)+\cdots+g_{n-1}(t)\big)\big)
& \text{if $t\in I$,} \\
0 & \text{otherwise.}\end{cases}$$ 
Thus    $\supp h_n\subseteq\supp\varphi_n\subseteq K_{n+2}$ is bounded. Put
  $$g_n:=W_{\lambda_n}(h_n)\in\Cc^\omega(\R)$$ where we take~$\lambda_n\in\R^>$ such that~$\dabs{g_n-h_n}_{r_n}<\delta_n$: any sufficiently large $\lambda_n$ will do, by Lemma~\ref{lem:bound lambda}.
The argument in \cite[Appendix~A]{fgh-an} then yields a function 
  $$g\colon I\to\R,\qquad g(t)=\sum_{i=0}^\infty g_i(t)\text{ for each $t\in I$}$$ 
  with $g\in \C^{r_n}(I)$  and $\dabs{f-g}_{L_n;\,r_n}<\varepsilon_n$ for each $n$.
Next set
 $$H_m := 2(\lambda_m/\pi)^{1/2}\dabs{h_m}(b_{m+2}-a_{m+2}) \in\R^{\ge},$$
and fix a sequence $(c_m)$ of positive reals such that $M:=\sum_{m\ge 1} c_m <\infty$. Then we can and do choose the sequences $(g_m)$, $(\lambda_m)$ so that in addition
$$H_m \exp(-\lambda_m/m)\ \leq\  c_m\quad\text{for all $m\ge 1$.}$$
Therefore
$$
\sum_n H_n\exp(-\lambda_n \rho) <\infty\qquad\text{for each $\rho\in\R^>$.}
$$
Now $g_n$ is the restriction to $\R$ of the entire function $\hat g_n$ given by
$$\hat g_n(z)\ =\ (\lambda_n/\pi)^{1/2}\int_{a_{n+2}}^{b_{n+2}} h_n(s)\ex^{-\lambda_n(s-z)^2}ds.$$ 
In the following we let $x=\Re z$, $y=\Im z$. Then $\Re (s-z)^2 = (s-x)^2-y^2$ and so
\begin{equation}\label{eq:hat gn(z)}
\abs{\hat g_n(z)} \le 
(\lambda_n/\pi)^{1/2} \int_{-\infty}^{\infty} \abs{h_n(s)}\ex^{-\lambda_n \Re(s-z)^2}\,ds
\le
\dabs{h_n}\ex^{\lambda_n y^2}.
\end{equation}
Put $$\rho_n\ :=\ \textstyle\frac{1}{2}\min\!\big\{ (a_n-a_{n+1})^2,(b_{n+1}-b_{n})^2\big\}\in\R^>$$ and
$$U_n \ :=\  \big\{z : a_{n+1}<\Re z<b_{n+1},\ 
\Re\!\big((z-a_{n+1})^2\big),\  \Re\!\big((z-b_{n+1})^2\big)>\rho_n
  \big\},$$
an open subset of $\mathbb C$ containing $K_n$ such that $\Re\!\big((s-z)^2\big)>\rho_n$
for all $z\in U_n$ and~$s\in \R\setminus K_{n+1}$. If $\rho_n\ge\rho_{n+1}$, then $U_n\subseteq U_{n+1}$.
The argument in \cite[Appendix~A]{fgh-an} shows that for each $n$ the series $\sum_m \hat g_m$ converges uniformly on compact subsets of~$U_n$, and so
we obtain a holomorphic function
$$z\mapsto \hat g(z):= \sum_m \hat g_m(z)\ \colon\ U=\bigcup_n U_n\to\Co$$
which extends $g$, as claimed.
Indeed, if $z\in U_n$ and $m\ge n+2$, then $\supp h_m\subseteq  K_{m+2}\setminus K_{m-1}\subseteq\R\setminus K_{n+1}$ and so
$\abs{\hat g_m(z)}\leq H_m\ex^{-\lambda_m\rho_n}$, hence   
\begin{equation}\label{eq:kn}
\abs{\hat g_m(z)}\leq  H_m\ex^{-\lambda_m/m}\le c_m\quad\text{ if
 $m\ge k_n:=\max\big\{\lceil 1/\rho_n\rceil, n+2\big\}$}
\end{equation}
and so
 $\sum_{m\ge k_n} \abs{\hat g_m(z)} \le M$.
 
In the proofs of Theorems~\ref{thm:WAP bd, 1} and~\ref{thm:WAP bd, 2} (given at the end of this section) we
need to control the growth of $\abs{\hat g(z)}$ for~$z\in U$, and this requires us to choose the quantities~$\delta_n$,~$\lambda_n$,~$c_n$ in a more explicit way.  With this in mind we  focus now in particular on the two cases~$I=\R$ and~$I=\R^>$,
  where we assume that~$a_n$,~$b_n$ are chosen as follows, for some $\delta\in\R^>$:
  \begin{enumerate}
\item[($\R$)\ \ ] $a_n=-b_n=-\delta n$ for each $n$;
\item[($\R^>$)] $a_n=\delta/(n+1)$, $b_n=\delta(n+1)$ for each $n$.
\end{enumerate}
Below we will label   various displayed statements pertaining to these two cases accordingly. For example,    note that
\begin{equation}\addtocounter{equation}{1}\tag{\theequation\ $\R$} \label{eq:rhon R}
\rho_n =  \delta^2/2,  \qquad U=\Co.
\end{equation}
This follows by observing that
given $z$, for sufficiently large $n$ we have~$a_{n+1}=-\delta(n+1)<x<\delta(n+1)=b_{n+1}$ as well as
$$(x-a_{n+1})^2-y^2-\rho_n=\big(x+\delta(n+1)\big)^2-y^2-(\delta^2/2)>0$$ and
$$(x-b_{n+1})^2-y^2-\rho_n=\big(x-\delta(n+1)\big)^2-y^2-(\delta^2/2)>0,$$ 
so $z\in U_n$.
We also have
\begin{equation}\tag{\theequation\ $\R^>$}\label{eq:rhon R>}
\rho_n =   \frac{\delta^2}{2}\left(\frac{1}{(n+1)(n+2)}\right)^2,\qquad U=\big\{z:\abs{\Im z}<\Re z\big\}.
\end{equation}
To see the latter note that  if   $x>0$ and $y^2<x^2$, then
$$(x-a_{n+1})^2-y^2-\rho_n \to x^2-y^2>0\text{ as $n\to\infty$,}$$
and thus $\Re\big((z-a_{n+1})^2\big)=(x-a_{n+1})^2-y^2>\rho_n$ for sufficiently large $n$, and likewise, 
$\Re\big((z-b_{n+1})^2\big)>\rho_n$ for all sufficiently large $n$.
Conversely, if $z\in U_n$, then~$x>a_{n+1}>0$ and $x^2-y^2\ge (x-a_{n+1})^2-y^2>\rho_n>0$. 

It will be convenient to assume, in addition to (i)--(iii) above:
\begin{enumerate}
\item[(iv)] $\varepsilon_{n}+\varepsilon_{n+2}\ge 2\varepsilon_{n+1}$ for each $n$, and
\item[(v)] $b_n-a_n=O(n)$.
\end{enumerate}
(Here (iv) will hold in the settings of the proofs of Theorems~\ref{thm:WAP bd, 1} and~\ref{thm:WAP bd, 2}, and~(v) clearly holds both in the cases ($\R$) and ($\R^>$).) We shall also assume $f\ne 0$ below, and  let~$x=\Re z$, $y=\Im z$.

We now estimate the growth of the derivatives of the $\varphi_n$.
Let~$C_n\in\R$ be as defined in Section~\ref{sec:bump fns}.  Suppose first we are in the case ($\R$).
By the remark after~\eqref{eq:hump fn, bd} we can take $c,d\in\R^{\ge 1}$
such that
\begin{equation}\addtocounter{equation}{1}\tag{\theequation\ $\R$}\label{eq:varphin, R}
\dabs{\varphi_n}_{m} \leq c\,m^{dm}\quad \text{for all $m$, $n$.}
\end{equation} 
In the case ($\R^>$) we 
note $a_n-a_{n+1}=\frac{\delta}{(n+1)(n+2)}$ for each~$n$, hence for $n\ge 1$ we 
have~$\dabs{\alpha_n}_{m} \leq (3/\delta)^{m}C_{m}\big((n+2)(n+3)\big)^{m}$ and 
$\dabs{\beta_n}_m\le (3/\delta)^m C_m$, so
$\dabs{\varphi_n}_m\le (3/\delta)^{m}C_{m}\big((n+2)(n+3)\big)^{m}$; we also have
$\dabs{\varphi_0}_m\le (18/\delta)^m C_m$.
Thus in this case we
obtain $c,d\in\R^>$ such that
\begin{equation}\tag{\theequation\ $\R^>$}\label{eq:varphin, R>}
\dabs{\varphi_0}_m \leq c\,m^{dm}, \qquad
\dabs{\varphi_n}_{m} \leq c\,(mn)^{dm}  \quad  \text{if $n\ge 1$.}
\end{equation}
Going forward we assume that we have real numbers $D_{mn}$ such that
$$D_{0n}=1,\quad 2^{m}\dabs{\varphi_n}_{m}\le   D_{mn}\le D_{m+1,n}, D_{m,n+1} \qquad\text{for all $m$, $n$.}$$ 
(In the case  ($\R$),  by  \eqref{eq:varphin, R}, for suitable~$c,d\in\R^{\ge 1}$ we can take~$D_{mn}:=c\,m^{dm}$. In the case~($\R^>$)
we may similarly take~$D_{m0}:=c\,m^{dm}$ and
$D_{mn}:=c\,(mn)^{dm}$ if~$n\ge 1$, for suitable $c,d\in\R^{\ge 1}$.)
Now put 
$$D_n:=D_{r_nn},\qquad N_n := 2^{n+1}D_n.$$
Then $N_n\ge 2D_{n}$,
thus $M_n=1+2^{r_n}\dabs{\varphi_n}_{r_n} \le 1+D_{n} \le N_n$. Also note:~$2N_n \le N_{n+1}$. 
{\it We now assume that in the construction of $\hat g$ above we chose}
$$\delta_n:=\frac{\varepsilon_n-\varepsilon_{n+1}}{4N_{n+1}}>0.$$
Note that these $\delta_n$ indeed satisfy the requirements  \eqref{eq:WAP, 0}: we have
$$\sum_{m=n}^\infty \delta_m M_{m+1}\le\sum_{m=n}^\infty \delta_m N_{m+1} =  \frac{1}{4} \sum_{m=n}^\infty (\varepsilon_m-\varepsilon_{m+1}) = 
\frac{1}{4}\lim_{k\to\infty} (\varepsilon_n-\varepsilon_{k+1})\le\frac{1}{4}\varepsilon_n,$$
and we also have $2\delta_{n+1}\le \delta_n$ since by (iv):
$$\frac{N_{n+2}}{2N_{n+1}} \ge 1 \ge \frac{\varepsilon_{n+1}-\varepsilon_{n+2}}{\varepsilon_{n}-\varepsilon_{n+1}}.$$
   In order to deduce from Lemma~\ref{lem:bound lambda} a lower bound on just how large~$\lambda_n$ needs to be taken,  
  we need an upper bound on~$\dabs{h_n}_m$ for $m\le r$, which we establish next.
 For this we note that
$\supp\varphi_n\subseteq K_{n+2}$ and $\dabs{g_n}_{m}\leq\dabs{h_n}_{m}$  for each $n$
yields
$$\dabs{h_n}_{m} \leq 2^{m}\dabs{\varphi_n}_{m}\big( \dabs{f}_{K_{n+2};\,m}+\dabs{h_0}_{m}+\cdots+\dabs{h_{n-1}}_{m}\big).$$
An easy induction on $n$ now shows that
$$\dabs{h_n}_{m} \le G_{mn}:=2^n (D_{mn})^{n+1}   \cdot\dabs{f}_{K_{n+2};\,m}.$$
In particular $\dabs{h_n}\le G_{0n}=2^n\dabs{f}_{K_{n+2}}$.
Also note that 
$$\dabs{f}_{K_{n+2}}\le G_{mn}\le G_{m+1,n},G_{m,n+1}.$$ Next put
$$\mu_n := 128 \sqrt{2}(\delta_n^{-1}G_{r_n+1,n})^3+1.$$
Then $\mu_{n+1}\ge\mu_n\ge 1$,
and  for each $\lambda\ge\mu_n$ we have~$\dabs{{W_\lambda(h_n)-h_n}}_{r_n}\leq\delta_n/2<\delta_n$ by Lemma~\ref{lem:bound lambda}. 
{\it Below we assume that in our construction of the sequences~$(\lambda_n)$,~$(g_n)$ we always chose~$\lambda_n=\mu_n$.}\/ Note that
$\delta_n\leq\delta_0\leq \varepsilon_0/2^{n+3}$ and so
$$\lambda_n \ge   (\delta_0^{-1}G_{0n})^3 \ge  8^{n+3} \varepsilon_0^{-3} G_{0n}^3.$$
Since  $f\ne 0$ by assumption,  we can take $n_0$ with $G_{0n_0}>0$. Then with $c:=\varepsilon_0^3/G_{0n_0}^2$ we have
$c\lambda_n \ge n^5 G_{0n}$ for each $n$.
The function $t\mapsto t^2\ex^{-t}\colon\R^>\to\R$ takes on its maximum value $4\ex^{-2}<1$ at $t=2$, so $t\ex^{-t}<t^{-1}$ for each $t>0$.
Hence for~$n\ge 1$:
$$\lambda_n^{1/2}\ex^{-\lambda_n/n} \le n\cdot (\lambda_n/n)\ex^{-\lambda_n/n}\le n^2/\lambda_n $$
and so
$\lambda_n^{1/2}\ex^{-\lambda_n/n}G_{0n} \le   c/n^3$.
We have $H_n =O\big( \lambda_n^{1/2} G_{0n}\, (n+2)\big)$ by (v),  
so we obtain a constant $c_*>0$ (not depending on the sequence $r_n$) such that
$$H_n \exp(-\lambda_n/n) \le   c_*/n^2\quad\text{for $n\ge 1$.}$$
{\it Hence in the construction of $\hat g$ above we may choose $c_n:=c_*/n^2$ for each $n\ge 1$.}\/

\subsection*{Proof of Theorem~\ref{thm:WAP bd, 1}}
Let $f$,  $\varepsilon$, $\rho$  
be as in the statement of Theorem~\ref{thm:WAP bd, 1}. 
We   choose~$a_n$,~$b_n$ as in   ($\R$) above, with $\delta:=\sqrt{2}$, and   set 
$\varepsilon_n:=\varepsilon(b_{n+1})={\varepsilon\big(\sqrt{2}(n+1)\big)}$ and $r_n:=\lfloor \rho(b_{n+1})\rfloor$. 
Then conditions~{(i)--(v)} hold, with~(iv)   a consequence of the convexity of~$\varepsilon$. We also have $\rho_n=1$ and $U=\Co$ by \eqref{eq:rhon R}.
We now construct the real entire function~$\hat g$ and its restriction~$g\in\C^\omega(\R)$ as described above, in the process choosing the quantities~$D_n$, $\delta_n$, $\lambda_n$, $c_n$ as indicated, in
the case ($\R$). (In particular, $c_n=c_*/n^2$ for $n\ge 1$, where the constant $c_*>0$ doesn't depend on $\rho$.)
Then 
for~$t\in L_n$ and~$k\leq \rho(\abs{t})$ we have~$\abs{t}\le b_{n+1}$, thus $k\le r_n$ and so
$$\abs{ (f-g)^{(k)}(t) } \leq \dabs{f-g}_{L_n;\,r_n} < \varepsilon_n=\varepsilon(b_{n+1})\le\varepsilon(\abs{t}).$$
Thus
$\abs{ (f-g)^{(k)}(t) } <\varepsilon(\abs{t})$ for all $t$ and $k\leq  \rho(\abs{t})$.
We now aim to estimate $\abs{\hat g(z)}$ when $\abs{z}\le t$. We need two lemmas. In the first one we 
bound
$$\lambda_{n}=128\sqrt{2}\big(\delta_{n}^{-1}G_{r_{n}+1,n}\big)^3+1$$
from above:

\begin{lemma}\label{lem:WAP bd, 1}
There is a $D\in\R^{\ge 1}$, independent of $f$,  $\varepsilon$, $\rho$, such that for each~$n$, 
$$\lambda_n \le \big(D(\rho(s)+1)\big)^{D(\rho(s)+1)s} \cdot \left(\frac{\dabs{f}_{s;\,\rho(s)+1}}{\Delta\varepsilon(s)} \right)^3+1\qquad\text{where $s:=\sqrt{2}(n+2)$.}$$
\end{lemma}
\begin{proof}
As we are in case ($\R$), we can assume that we have $c,d\in\R^{\ge 1}$ (independent of $f$, $\varepsilon$, $\rho$) such that $D_{mn}=cm^{dm}$ for all~$m$,~$n$.   
We have
$$\delta_{n}^{-1}=\frac{2^{n+3}D_{n+1}}{\varepsilon_{n}-\varepsilon_{n+1}}\qquad\text{where $D_{n+1}=D_{r_{n+1},n+1}$.}$$
Since $\varepsilon$ is convex and decreasing,  
$$\varepsilon_{n}-\varepsilon_{n+1} 	=\varepsilon(s-\sqrt{2})-\varepsilon(s) \ge \varepsilon(s)-\varepsilon(s+1)= \Delta\varepsilon(s).$$  
Now $r_{n+1} \le \rho(s)$, thus
$$\delta_{n}^{-1} = \frac{c2^{n+3}(r_{n+1})^{dr_{n+1}}}{\varepsilon_{n}-\varepsilon_{n+1}} \le \frac{2c\,2^{s/\sqrt{2}}\rho(s)^{d\rho(s)}}{\Delta\varepsilon(s)}$$
and so 
$$\delta_{n}^{-3} \leq \frac{c_0\,2^{(3/\sqrt{2})s} \rho(s) ^{d_0\rho(s)} }{\Delta\varepsilon(s)^3}\qquad\text{with $c_0:=(2c)^3$, $d_0:=3d$.}$$
Moreover, using that $\rho$ is increasing:
\begin{align*}
G_{r_{n}+1,n}	&=		2^{n} (D_{r_{n}+1,n})^{n+1}\cdot\dabs{f}_{K_{n+2;\,r_{n}+1}} \\
					&\leq 2^{n}\big(c(r_{n}+1)\big)^{d(r_{n}+1)(n+1)}\cdot \dabs{f}_{K_{n+2;\,r_{n}+1}} \\
					&\leq 2^{s/\sqrt{2}-2}\big(c(\rho(s)+1)\big)^{d(\rho(s)+1)(s/\sqrt{2}-1)}\cdot \dabs{f}_{s;\,\rho(s)+1} \\
					&\leq	  \big(2c(\rho(s)+1)\big)^{d(\rho(s)+1)(s/\sqrt{2}-1)}\cdot \dabs{f}_{s;\,\rho(s)+1}. 
					\end{align*}
So with $c_1:=2c$, $d_1:=3d/\sqrt{2}$  we have
$$(G_{r_{n}+1,n})^3 \le   \big(c_1\big(\rho(s)+1\big)\big)^{d_1(\rho(s)+1)s}\cdot \dabs{f}_{s;\,\rho(s)+1}^3.$$
Combining these estimates for $\delta_{n}^{-3}$ and $(G_{r_{n}+1,n})^3$   yields $D$ with the required properties.
\end{proof}

\noindent
In the following we set $U_{-1}:=\emptyset$. 
 
\begin{lemma}\label{lem:n vs z}
Suppose   $z\notin U_{n-1}$.  Then
$\sqrt{2} n\leq \abs{x}+\sqrt{1+y^2}\le 1+ \sqrt{2}\abs{z}$.
\end{lemma}
\begin{proof} The first inequality is clear if $\abs{x}\ge \sqrt{2} n$, so suppose $\abs{x}<\sqrt{2} n$. Then~${n\ge 1}$, and thus~$z\notin U_{n-1}$ yields~$\Re (z-\sqrt{2} n)^2 \le 1$ or $\Re(z+\sqrt{2} n)^2 \le 1$.
In the first case~$(\sqrt{2} n-x)^2\le 1+y^2$ where $\sqrt{2} n-x>0$ and thus $$\sqrt{2} n\leq x+\sqrt{1+y^2}\le \abs{x}+\sqrt{1+y^2}.$$
In the second case, similarly   $(\sqrt{2} n+x)^2\le 1+y^2$ where $\sqrt{2} n+x>0$ and thus~$\sqrt{2} n\le -x+\sqrt{1+y^2}\le \abs{x}+\sqrt{1+y^2}$.
The second inequality is a consequence of Lemma~\ref{lem:sqrt ineq}.\end{proof}

\noindent
Now take $n$ such that $z\in U_n\setminus U_{n-1}$. From~\eqref{eq:kn} and the remark after it  recall that~$\sum_{m\ge k_n} \abs{\hat g_m(z)}\le M$ where $k_n=n+2$,
thus
$$\abs{\hat g(z)} \le \sum_{m\le n+1} \abs{\hat g_m(z)} + \sum_{m\ge n+2} \abs{\hat g_m(z)}$$
where the second (infinite) sum is~$\le M$, so we focus on the first sum. 
With~$s:=\sqrt{2}t+1$ we have $\sqrt{2} n\le s$ by Lemma~\ref{lem:n vs z}, so $n\le s/\sqrt{2}\le t+1\le s$.
Recalling that~$K_{n+3}=\big[{-\sqrt{2}(n+3)},\sqrt{2}(n+3)\big]$,
we get $G_{0,n+1}=2^{n+1}\dabs{f}_{K_{n+3}} \leq 2^{s+1}\dabs{f}_{s+3\sqrt{2}}$, hence
by~\eqref{eq:hat gn(z)}, for~$m\le n+1$:
$$\abs{\hat g_m(z)} \leq G_{0m}\cdot\exp(\lambda_m  t^2) \leq 
G_{0,n+1}\cdot \exp(\lambda_{n+1}  t^2) \leq 
2^{s+1}\dabs{f}_{s+3\sqrt{2}}  \cdot\exp(\lambda_{n+1}  t^2).$$
This yields
$$\sum_{m\le n+1} \abs{\hat g_m(z)} \le (s+2)\cdot 2^{s+1}\dabs{f}_{s+3\sqrt{2}}\cdot\exp(\lambda_{n+1}  t^2),$$%
hence
\begin{align*}
\dabs{\hat g}_t	&\leq  (s+2) \cdot 2^{s+1}\dabs{f}_{s+3\sqrt{2}}\cdot\exp(\lambda_{n+1}  t^2) + M \\
				&\leq \exp\big( (s+2)+\log^+\dabs{f}_{s+3\sqrt{2}}+\lambda_{n+1}t^2\big) + M\\
				&\leq \exp\big( N + \log^+\dabs{f}_{s+3\sqrt{2}}+  \lambda_{n+1}s^2 \big) \quad\text{where $N:=\log(1+M)+3$,} \\
				&\leq \exp\big( N\cdot s^2\cdot(1 + \log^+\dabs{f}_{s+3\sqrt{2}} + \lambda_{n+1})\big).
\end{align*}
Next take $D$ as in Lemma~\ref{lem:WAP bd, 1}. Since $\rho$ and $1/\Delta\varepsilon$ are increasing and $s+3\sqrt{2}\geq\sqrt{2}(n+3)$, we obtain
$\lambda_{n+1}\le\lambda(s+3\sqrt{2})+1$ where
$$\lambda(s):= \big(D(\rho(s)+1)\big)^{D(\rho(s)+1)s} \cdot \left(\frac{\dabs{f}_{s;\,\rho(s)+1}}{\Delta\varepsilon(s)} \right)^3.$$
Hence with $C:=2N$   we have
$$\dabs{\hat g}_t \le \exp\big( C\cdot s^2\cdot\big(1+\log^+ \dabs{f}_{s+3\sqrt{2}}+\lambda(s+3\sqrt{2}) \big)\big)\quad\text{for $t\ge 0$, $s=\sqrt{2}t+1$.}$$
Here $C$ does not depend on   $\rho$, and $D$ does not depend on $f$, $\varepsilon$, $\rho$.
This concludes the proof of the theorem. \qed

\subsection*{Proof of Theorem~\ref{thm:WAP bd, 2}}
Let $\alpha$, $f$, $\varepsilon$, $\rho$, $V$ be as in  Theorem~\ref{thm:WAP bd, 2}, so
$$V= \big\{ z: \Re z>0,\ (\Im z)^2 \leq (\Re z)^2-\alpha \big\}.$$
We choose $(a_n)$, $(b_n)$ as in ($\R^>$)   with $\delta:=\sqrt{\alpha/2}$, and
we set $\varepsilon_n:=\varepsilon(b_{n+1})=\varepsilon\big(\delta(n+2)\big)$ and
$r_n:=\lfloor \rho(b_{n+1}) \rfloor$ for each $n$. Then (i)--(v) hold, and by \eqref{eq:rhon R>} we have~$V\subseteq U = {\big\{z: \abs{\Im z}<\Re z\big\}}$ and
$$\rho_n = \frac{\alpha}{4}\left(\frac{1}{(n+1)(n+2)}\right)^2.$$
We now construct the holomorphic function $\hat g\colon U\to\Co$ with $\hat g(\R^>)\subseteq\R$ as described before,
and set $g:=\hat g|_{\R^>}\in\C^\omega(\R^>)$. 
As in the proof of Theorem~\ref{thm:WAP bd, 1}, we see that for $t\in L_n$  and $k\le\rho(t)$ we 
have $\abs{ (f-g)^{(k)} } < \varepsilon(t)$.
Since $\bigcup_n L_n = \R^>$, this yields~$\abs{ (f-g)^{(k)}(t) } <\varepsilon(t)$ for all $t>0$ and $k\leq  \rho(t)$.
Now put $V_n:=U_n\cap V$, so~$V=\bigcup_n V_n$, and set $V_{-1}:=\emptyset$.

\begin{lemma}\label{lem:n vs z, R>}
Let $z\in V\setminus V_{n-1}$. Then
$n \leq (2/\sqrt{\alpha})\abs{z}$.
\end{lemma}
\begin{proof}
We may assume $n\ge 1$. Note that   $z\in V$ gives $x^2\ge y^2+\alpha\ge\alpha$ and $x>0$,
so $x\ge\sqrt{\alpha}>\delta>a_n$. If $x\ge b_n=\delta(n+1)$, then~$n < {n+1}\le (1/\delta)x$, hence~$n \leq (2/\sqrt{\alpha})\abs{z}$;
thus suppose $x<b_n$.
We first show that~$\Re\big( {(z-a_n)^2} \big)   > \rho_{n-1}$.
To see this note   $z\in V$ also yields $y^2\le x^2-\alpha$, and together with~$a_n x < a_nb_n=\delta^2=\alpha/2$, 
\begin{align*}
\Re\big( (z-a_n)^2 \big) 	= (x-a_n)^2-y^2 
							&\ge  a_n^2-2a_nx+\alpha \\
							&> a_n^2= \frac{\delta^2}{(n+1)^2} \\
							&\ge \frac{\delta^2}{2} \left(\frac{1}{n(n+1)}\right)^2 = \rho_{n-1}
\end{align*}
as claimed. Now $z\notin V_{n-1}$ yields $\Re\big((z-b_n)^2\big) \leq \rho_{n-1}$, and so
$$\big(x-\delta(n+1)\big)^2-y^2 = \Re\big((z-b_n)^2\big) \le \rho_{n-1}\le  \alpha/16,$$
hence $\delta(n+1)-x\le \sqrt{ (\alpha/16) + y^2}$. Using Lemma~\ref{lem:sqrt ineq} we get
$\delta(n+1)\le (\sqrt{\alpha}/4)+\sqrt{2}\abs{z}$ and thus
$\delta n < \sqrt{2}\abs{z}$, hence $n \leq (2/\sqrt{\alpha})\abs{z}$ follows.
\end{proof}

\begin{lemma}\label{lem:D}
There is a $D\in\R^{\ge 1}$, independent of $f$, $\varepsilon$, $\rho$, such that for each $n\ge 1$,
$$\lambda_n \le \big(D(\rho(s)+1)\big)^{D(\rho(s)+1)s^2} \cdot \left(\frac{\dabs{f}_{s;\,\rho(s)+1}}{\Delta\varepsilon(s)} \right)^3+1\qquad\text{where $s:=\delta(n+3)$.}$$
\end{lemma}
\begin{proof}
Since we are in the case ($\R^>$), we can assume to  have $c,d\in\R^{\ge 1}$ (independent of $f$, $\varepsilon$, $\rho$) such that 
 $D_{mn}=c(mn)^{dm}$ for all $m$ and $n\ge 1$.  
 Suppose~$n\ge 1$, and recall again that
$$\delta_{n}^{-1}=\frac{2^{n+3}D_{n+1}}{\varepsilon_{n}-\varepsilon_{n+1}}\qquad\text{where $D_{n+1}=D_{r_{n+1},n+1}$.}$$
Using Lemma~\ref{lem:convex diff quot}, from the convexity of $\varepsilon$ we obtain
$$\varepsilon_n-\varepsilon_{n+1} = \varepsilon(s-2\delta)-\varepsilon(s-\delta)\ge \delta\big(\varepsilon(s)-\varepsilon(s+1)\big)=
\delta\cdot\Delta\varepsilon(s).$$
We also have $r_{n+1}\le \rho(b_{n+2})=\rho(s)$ and thus
$$\delta_n^{-1} \le \frac{c2^{s/\delta}\big(\rho(s)(s/\delta-2)\big)^{d\rho(s)}}{\varepsilon_n-\varepsilon_{n+1}} \leq
\frac{(c/\delta)\, 2^{s/\delta} \, \big(\rho(s)(s/\delta)\big)^{d\rho(s)}}{\Delta\varepsilon(s)},$$
so
$$\delta_n^{-3} \le \frac{c_0 2^{(3/\delta)s} \big( \rho(s)(s/\delta) \big)^{d_0\rho(s)}}{\Delta\varepsilon(s)^3}
\qquad\text{with $c_0:=(c/\delta)^3$, $d_0:=3d$.}$$
Moreover,  
\begin{align*}
G_{r_n+1,n} 	&= 2^n (D_{r_n+1,n})^{n+1}\cdot \dabs{f}_{K_{n+2};\,r_{n}+1} \\
				&\leq 2^n \big(c(r_n+1)n\big)^{d(r_n+1)(n+1)}\cdot\dabs{f}_{K_{n+2};\,r_{n}+1} \\
				&\leq  \big(2c(\rho(s)+1)\big)^{d(\rho(s)+1)(s/\delta-3)(s/\delta-2)}\cdot \dabs{f}_{s;\,\rho(s)+1} 
\end{align*}
and thus
$$(G_{r_n+1,n})^3 \leq \big(c_1\big(\rho(s)+1\big)\big)^{d_1(\rho(s)+1)s^2}\cdot \dabs{f}_{s;\,\rho(s)+1}^3$$
where $c_1:=2c$, $d_1:=3d/\delta^2$, and this yields the lemma.
\end{proof}

\noindent
Recall that
$$k_n = \max\big\{\lceil 1/\rho_n\rceil, n+2 \big\} \le \max\left\{   (4/\alpha)\big( (n+1)(n+2)\big)^2  , n+1 \right\}+1,$$
so  we obtain a constant $c_0\in\R^{>}$, only depending on $\alpha$, such that
$$k_n+2  \le   c_0 (t^2+1)    \qquad\text{for all $n$ and $t\ge (\delta/\sqrt{2}) n$.}$$
Let now $z\in V$ with $\abs{z}\le t$; we aim to estimate $\abs{\hat g(z)}$ in terms of $t$.
Take $n$ with~$z\in V_n\setminus V_{n-1}$, so $t \ge (\delta/\sqrt{2})n$ by Lemma~\ref{lem:n vs z, R>}.
We have $\sum_{m\ge k_n} \abs{\hat g(z)} \le M$ by \eqref{eq:kn}.
Also $K_{m+2}=\big[\delta/(m+3),\delta(m+3)\big]$ and so
$G_{0m}=2^m\dabs{f}_{K_{m+2}} \leq 2^m \dabs{f}_{\delta(m+3)}$, hence if $m<k_n$ then
$G_{0m}\leq 2^{c_0(t^2+1) } \dabs{f}_{c_1(t^2+1) }$ where $c_1:=\delta c_0$. This yields
\begin{align*}
\sum_{m<k_n} \abs{\hat g_m(z)} 	&\le k_n\cdot 2^{c_0(t^2+1)} \dabs{f}_{\delta c_0(t^2+1)}\cdot\exp(\lambda_{k_n-1}t^2) \\
								&\le c_0(t^2+1)\cdot 2^{c_0(t^2+1)} \dabs{f}_{ c_1(t^2+1)}\cdot\exp(\lambda_{k_n-1}t^2),
\end{align*}
and hence for suitable $N\in\R^>$, not depending on $\rho$:
\begin{align*}
\abs{\hat g(z)} &\leq \sum_{m<k_n} \dabs{\hat g_m(z)} + \sum_{m\ge k_n}  \dabs{\hat g_m(z)} \\
				 &\leq c_0(t^2+1) \cdot 2^{c_0(t^2+1)} \dabs{f}_{\delta c_0(t^2+1)}\cdot\exp(\lambda_{k_n-1}t^2) + M \\
				 &\leq \exp\big(  N\cdot (t^2+1) \cdot (1 + \log^+ \dabs{f}_{c_1(t^2+1)}+\lambda_{k_n-1})\big). 
\end{align*}
Now put  $C:=2N/c_1$, take $D$ as in Lemma~\ref{lem:D}, and let 
$$s:=c_1(t^2+1), \qquad
\lambda(s):=\big(D(\rho(s)+1)\big)^{D(\rho(s)+1)s^2} \cdot \left(\frac{\dabs{f}_{s;\,\rho(s)+1}}{\Delta\varepsilon(s)} \right)^3.$$
Then   $\lambda_{k_n-1}\le\lambda(s)+1$, and  so
$\abs{\hat g(z)} \le \exp\big( C\cdot s\cdot\big(1+\log^+ \dabs{f}_{s}+\lambda(s) \big)\big)$.
Here $C$ does not depend on   $\rho$, and $c_1$, $D$ only     on $\alpha$.
This concludes the proof of   Theorem~\ref{thm:WAP bd, 2}. \qed

 \end{document}